\documentclass[11pt]{amsart}
\usepackage{amsmath,amsthm,amssymb,graphicx,enumerate,natbib,url,color,bbm}
\newtheorem{theorem}{Theorem}[section]
\newtheorem{lemma}[theorem]{Lemma}

\theoremstyle{remark}
\newtheorem{remark}[theorem]{Remark}

\numberwithin{equation}{section}
\allowdisplaybreaks
\DeclareMathOperator{\sgn}{sgn}

\begin{document}

\author[S. Ghosh]{Souvik Ghosh}
\address{Department of Statistics\\
Columbia University \\
New York, NY 10027.}
\email{ghosh@stat.columbia.edu}

\author[S. Resnick]{Sidney Resnick}
\address{School of Operations Research and Information Engineering,
  Rhodes Hall\\ Cornell University\\ Ithaca, NY 14853.}
\email{sir1@cornell.edu}

\date{\today}

\title[Mean Excess Plots]{A Discussion on Mean Excess Plots}  

\thanks{S. Ghosh was partially supported by the FRA program at Columbia University. S. Resnick was partially supported by ARO Contract W911NF-07-1-0078 at Cornell University.  The authors are grateful to the  referees the editors for their helpful and extensive comments.}

\begin{abstract} 

 A widely used  tool in the study of risk, insurance and extreme
 values  is the mean excess plot. One use is for  validating a
 generalized Pareto model for the excess distribution. This paper
  investigates some theoretical and practical aspects of the use of
  the mean
 excess plot. 
  \end{abstract}

\maketitle

\begin{section}{Introduction} \label{sec:intro}

The distribution of the excess  over a threshold $ u$ for a random variable $ X$ with distribution function $ F$ is defined as 
\begin{equation}\label{eq:excessdistn}
	F_{ u}(x)=P\big[  X-u\le x|X>u \big].
\end{equation}
This excess distribution is the foundation for peaks over threshold
(POT) modeling \citep{
embrechtskluppelbergmikosch:1997,
coles:2001}
which fits appropriate distributions to data of excesses. 
The use of peaks over threshold modeling is widespread and
applications include:
\begin{itemize} 
        \item Hydrology:  It is critical  to model the level of water
          in a river or sea to avoid flooding. The level $u$ could
          represent the height of a dam, levee or river bank. See
          \cite{todorovic1970smf} and \cite{todorovic1971spf}. 
        \item Actuarial science: 
Insurance companies set premium levels based on models
          for large losses. Excess of loss insurance pays for losses
          exceeding a contractually agreed amount. See
          \cite{hogg1984ld}, \cite{embrechtsmcneilfrey:2005}. 
        \item Survival analysis: The POT method is used for
          modeling lifetimes; see \cite{guess1985mrl}. 
        \item Environmental science: 
Public health agencies set standards for pollution
  levels. Exceedances of these standards generate public alerts or corrective measures; see \cite{smith1989eva}.
\end{itemize}

Peaks over threshold modeling is based on the 
generalized Pareto class of distributions  
being appropriate for describing statistical properties of excesses.
A random variable $ X$ has a generalized Pareto distribution (GPD) if it has a  cumulative distribution function of the form
\begin{equation}\label{eq:GPD}
	G_{ \xi ,\beta }(x)= \left\{   \begin{array}{ll}  1-(1+  \xi x/ \beta  )^{ -1/\xi }  &\mbox{ if }\xi \neq 0 \\ 1-\exp(-x/\beta ) & \mbox{ if } \xi =0 \end{array}  \right.
\end{equation}
where $ \beta >0$, and $ x\ge 0$ when $ \xi \ge 0$ and $ 0\le x\le -\beta/\xi $ if $ \xi <0$. The parameters $ \xi $ and $ \beta $ are referred to as the \emph{shape} and \emph{scale} parameters respectively. For a Pareto distribution, the tail index $ \alpha $ is just the reciprocal of $ \xi$ when $ \xi>0$. A special case is when $ \xi =0$ and in this case the GPD is the same as the exponential distribution with mean $ \beta $. 

The Pickands-Balkema-de Haan Theorem 
\cite[Theorem 7.20, page 277]{embrechtsmcneilfrey:2005}
provides the theoretical justification for the centrality of 
the GPD class of distributions for peaks over threshold modeling.
This result shows that for a large class of distributions (those
distributions in a maximal domain of attraction of the extreme value
laws), the excess distribution $F_u$ is asymptotically equivalent to a
GPD law $G_{\xi, \beta(u)}$, as the threshold $u$ appropaches the
right endpoint of the distribution $F$.  Here the asymptotic shape
parameter $\xi$ is fixed but the scale $\beta (u) $ may depend on $u$.
More precise statements are given below in Theorems
\ref{thm:char:pos:xi},
\ref{thm:char:neg:xi}, and
\ref{thm:char:zero:xi}.
 For this reason the GPD is a
natural candidate for modeling peaks over a threshold.

The choice of the extreme threshold $u$, where the GPD model provides
a suitable approximation to the excess distribution $F_u$ is critical
in applications. The {\it mean excess (ME)\/} function is a popular
tool used to aide this choice of $u$ and also to determine the
adequacy of the GPD model in practice. The ME function of a random
variable $X$ is defined as:
\begin{equation}\label{eq:meanexcess}
	M(u):= E\big[  X-u|X>u \big],
\end{equation}
provided $EX_+<\infty,$
and is also known as the {\it mean residual life function\/},
especially in
 survival analysis. It has been studied as early as
\cite{Benktander:1960}. See  \cite{hall1981mrl} for a
discussion of properties of mean excess functions. 
 Table 3.4.7 in \cite[p.161]{embrechtskluppelbergmikosch:1997} gives 
the mean excess function for some standard distributions. 

Given an independent and identically distributed (iid) sample $X_1,\dots,X_n$ from
$F(x)$, a natural estimate of  $M(u)$ 
is the empirical ME function
$ \hat M(u)$ defined as
\begin{equation}\label{eq:empiricalmeanexcess}
	\hat M(u)= \frac{ \sum_{ i=1}^{ n}(X_{ i}-u)I_{ [ X_{ i}>u]}}{ \sum_{ i=1}^{ n}I_{ [ X_{ i}>u]}}, \ \ \ \ u\ge 0.
\end{equation}
  \cite{yang1978ebf} suggested
 the use of the empirical ME function and
established the uniform strong consistency of 
$	\hat M(u)$ over compact $u$-sets; that is, for any $ b>0$
\begin{equation}\label{eqn:sid1}
 	P\Big[ \lim_{ n\to \infty} \sup_{ 0\le u\le b} \big| \hat M(u)-M(u) \big|=0  \Big]=1.    
\end{equation}

In the context of extremes, however, \eqref{eqn:sid1} is not
especially informative since what is of interest is the behavior of
$\hat M(u)$ in a neighborhood of the right end point of $F$, which
could be $\infty$. In this case the GPD plays a pivotal role. For a
random variable $ X\sim G_{ \xi ,\beta }$ , we have $E(X)<\infty$ iff
$ \xi <1$ and in this case, the ME function of $X$ is linear in $u$:
\begin{equation}\label{eq:megdp}
	M(u)= \frac{  \beta}{1-\xi  }+ \frac{ \xi}{1-\xi}u,
\end{equation}
where $ 0\le u< \infty $ if $ 0\le \xi<1$ and $ 0\le u\le -\beta /\xi$
if $ \xi<0$. 
 In fact, the linearity of the mean excess function
characterizes the GPD class.
See 
\cite{embrechtsmcneilfrey:2005, embrechtskluppelbergmikosch:1997}. 
   \cite{davison1990meo} used this
property to
devise a simple graphical check that  data conforms to
a GPD model; their method is based on the  ME plot which is the plot
of the points $ \{   (X_{ (k)},\hat M(X_{ (k)})):1< k\le n  \}$, where
$ X_{ (1)}\ge X_{ (2)}\ge \cdots\ge X_{ (n)}$ are the order statistics
of the data. If the ME plot  is close to linear for
high values of the threshold  then there is no evidence against use
of a GPD
model.  See also \cite{embrechtskluppelbergmikosch:1997} and
\cite{hogg1984ld} for the implementation of this plot in practice. 

In this paper we establish the asymptotic behavior of the ME plots for
large thresholds. We assume $F$ is in the maximal domain of attraction
of an extreme value law with shape parameter $\xi$. When $\xi <1, $
we show that, as expected, for high thresholds the ME plot viewed as a
random closed set converges in the Fell topology to a straight line.
A novel aspect of our study is we also
consider  the ME plot in the case $\xi>1$, the case where the ME function
does not exist, and show that the ME plot converges to a {\it random\/}
curve. This also holds in the more delicate case $\xi=1$ after suitable
rescaling. These results show that the ME plot is inconsistent when
$\xi \geq 1$ and emphasizes that knowledge of a finite mean is
required.

It is tempting to argue that consistency of the ME plot $\hat M(u)$
should imply, by a continuity argument, the consistency of the estimator of $\xi$
obtained from computing the slope of the line fit to the ME
plot. However, this slope functional is not necessarily continuous as
discussed in \cite{das2008qpr}. So consistency of the slope
function requires further work and is an ongoing investigation.

The paper is arranged as
follows. In Section 
\ref{sec:back} 
we  briefly discuss required background on convergence of random
closed sets
and then study the ME plot in Section \ref{sec:main}.
 In Section \ref{sec:exmethod}
we discuss advantages and disadvantages  of the mean excess plot
and how this tool 
compares with  {other techniques} of extreme value theory such as  the Hill
estimator, the Pickands estimator and the QQ plot. We illustrate  the
behavior of the empirical mean excess plot for some simulated data sets in
Section \ref{sec:simul} and in 
Section \ref{sec:data} we  analyze three
real data sets obtained from different subject areas  and 
also compare different tools.
\end{section}

\section{Background }\label{sec:back} 
\subsection{Topology on closed sets of $ \mathbb{R}^{
      2}$}\label{subsec:top}
Before we start any discussion on whether a mean excess plot is a reasonable diagnostic tool we need to understand what it means to talk about convergence of plots. So we discuss the topology on a set containing the plots.

We denote the collection of closed subsets of $ \mathbb{R}^{ 2}$ by $ \mathcal{F}$.  We consider a hit and miss topology on $ \mathcal{F}$ called the Fell topology. The Fell topology is generated by the families$ \{   \mathcal{F}^{ K}, K \mbox{ compact}  \}$ and $ \{   \mathcal{F}_{ G}, G\mbox{ open}  \}$ where for any set $ B$
\[ 
 	\mathcal{F}^{ B}= \{   F\in \mathcal{F}:F\cap B  =\emptyset \} \ \ \mbox{ and } \ \ \mathcal{F}_{ B}=\{   F\in \mathcal{F}:F\cap B \neq \emptyset  \} 
\]
So $ \mathcal{F}^{ B}$ and $ \mathcal{F}_{ B}$ are collections of closed sets which miss and hit the set $ B$, respectively. This is why such topologies are called hit and miss topologies.  In the Fell topology a sequence of closed sets $ \{   F_{ n}  \}$ converges to $ F\in \mathcal{F}$ if and only if the following two conditions hold:
\begin{itemize} 
        \item $ F$ hits an open set $ G$ implies there exists $ N\ge 1$ such that for all $ n\ge N$, $ F_{ n}$ hits $ G$. 
        \item $ F$ misses a compact set $ K$ implies there exists $ N\ge 1$ such that for all $ n\ge N$, $ F_{ n}$ misses $ K$. 
\end{itemize}
The Fell topology on the closed sets of $ \mathbb{R}^{ 2}$ is
metrizable and we indicate convergence in this topology of a sequence $\{F_n\}$ of
closed sets to a limit closed set $F$ by $F_n \to F$.
Sometimes, rather than work with the topology, it is easier to deal with the following characterization of convergence.
\begin{lemma}\label{lem:setconv} 
 A sequence $ F_{ n}\in\mathcal{F}$ converges to  $ F\in\mathcal{F}$ in the Fell topology if and only if the following two conditions hold:
\begin{enumerate}
\item For any  $ t\in F$  there exists $   t_{ n}\in F_{ n}$  such that $  t_{ n}\to t.$
\item If for some subsequence $ (m_{ n})$, $ t_{ m_{ n}}\in F_{ m_{ n}}$ converges, then $ \lim\limits_{ n\to \infty} t_{ m_{ n}}\in F$.
\end{enumerate}
\end{lemma}

See Theorem 1-2-2 in \cite[p.6]{matheron1975rsa} for a proof of this
Lemma.
Since the topology is metrizable, the  definition of convergence in
probability is obvious. The following result is a well-known and
helpful characterization for convergence in probability of random
variables and it holds for random sets  as well; see Theorem 6.21 in
\cite[p.92]{molchanov:2005}. 

\begin{lemma}\label{lem:convprob} 
 A sequence of random sets $ (F_{ n})$ in $ \mathcal{F}$ converges in probability to a random set $ F$ if and only if for every subsequence $ (n^{ \prime })$ of  $ \mathbb{Z}_{ +}$ there exists a further subsequence $ (n^{ \prime \prime })$ of $ (n^{ \prime })$ such that $ F_{ n^{ \prime \prime }} \to F$-a.s.
\end{lemma}

We use the following notation: For a real number $ x$ and a
set $ A\subset \mathbb{R}^{ n}$, $xA=\{   xy:y\in A  \} $.

\cite{matheron1975rsa} and \cite{molchanov:2005} are good
references for the theory of random sets.

\subsection{Miscellany.}\label{subsec:misc}
Throughout this paper we will take $ k:=k_{ n}$ to be a sequence
increasing to infinity such that $ k_{ n}/n\to 0$.  For a distribution
function $F(x)$ we write $\bar F(x)=1-F(x) $ for the tail and the
quantile function is 
$$F^\leftarrow (1-\frac 1u)=\inf\{s:F(s)\geq
1-\frac 1u\}=\Bigl(\frac{1}{1-F}\Bigr)^\leftarrow (u).$$
  A function
$U:(0,\infty)\mapsto \mathbb{R}_+$ is regularly varying with index
$\rho \in \mathbb{R}$, written $U\in RV_\rho$, if
$$\lim_{t\to\infty} \frac{U(tx)}{U(t)}=x^\rho,\quad x>0.$$
We denote  the space of nonnegative Radon measures $ \mu$ on $
(0,\infty]$ metrized by the vague metric by $ M_{ +}(0,\infty]$.  Point measures are written as a function of their points $\{x_i,
i=1,\dots,n\}$ by $\sum_{i=1}^n \delta_{x_i}.$ See, for example,
\cite[Chapter 3]{resnick:1987}.

{We will use the following notations to denote different classes of functions: For $  0\le a<b\le \infty$
\begin{enumerate}[(i)] 
        \item  $ D[a,b)$: Right-continuous functions with finite left limits defined on $ [a,b)$.
        \item $ D_{ l}[a,b)$: Left-continuous functions with finite right limits defined on $ [a,b)$.
\end{enumerate}
We will assume that these spaces are equipped with the Skorokhod topology and the distance function. In some cases we will also consider product spaces of functions and then the topology will be the product topology. For example, $ D^{ 2}_{ l}[1,\infty)$ will denote the class of 2-dimensional functions on $ [1,\infty)$. The classes of functions defined on the sets $ [a,b]$ or $ (a,b]$ will have the obvious notation.  }

{\begin{section}{Mean Excess Plots}\label{sec:main}
As discussed in the introduction, a random variable having $ G_{
  \xi,\beta }$ distribution with $ \xi<1$ has a linear ME function
given by \eqref{eq:megdp} where the slope $\xi$  is positive $(0<\xi<1),$ negative or
 $\xi=0$. We consider these three cases separately. 

\begin{subsection}{Positive Slope}\label{subsec:posxi}
In this subsection we concentrate on the case where $ \xi>0$. {A finite
mean for $F$ is guaranteed when $\xi<1$ and we also investigate what
happens when $\xi \geq 1$.}

The following Theorem is a combination of Theorem 3.3.7 and Theorem 3.4.13(b) in \cite{embrechtskluppelbergmikosch:1997}.
\begin{theorem}\label{thm:char:pos:xi} 
Assume $ \xi>0$. The following are equivalent for a cumulative distribution function $ F$:
 \begin{enumerate} 
        \item  $ \bar F \in RV_{ -1/\xi}$, i.e., for every $ t>0$ 
$
 	\lim_{ x \rightarrow \infty }\frac{ \bar F(tx)}{ \bar F(x)}= t^{ -1/\xi}.   
$
\item $ F$ is in the maximal  domain of attraction of  a Frechet distribution with parameter $ 1/\xi$, i.e.,
\[ 
 	\lim_{ n \rightarrow \infty  } F^{ n}(c_{ n}x)=\exp\{  - x^{ -1/\xi}  \} \ \ \ \ \mbox{ for all } x>0.   
\] 
where $ c_{ n}=F^{ \leftarrow }(1-n^{ -1})$.
\item There exists a positive measurable function $ \beta(u)$ such that 
\begin{equation}\label{eq:pick:Bal:deH} 
 	\lim_{ u \rightarrow \infty }\sup_{x\ge u } \big| F_{ u}(x)-G_{ \xi,\beta (u)}(x) \big|=0.    
\end{equation}
\end{enumerate}
\end{theorem}

Theorem \ref{thm:char:pos:xi}(3) is one case of the
 Pickands-Balkema-de Haan theorem. It guarantees the existence of a
 measurable function $ \beta (u)$ for which \eqref{eq:pick:Bal:deH}
 holds but does not construct this function. However, $\beta (u)$  can be obtained
 from  Karamata's representation of a regular varying
 function \citep{bingham:goldie:teugels:1989}, namely if $\bar F \in
 RV_{-1/\xi}$,  there exists $ 0<z<
 \infty$ such that  
\[ 
 	\bar F(x)   = c(x)\exp\left\{\int_{ z}^{ x} \frac{ 1}{a(t) }dt\right\} \ \ \ \ \mbox{ for all } z<x< \infty
\]
where $ c(x)\to c>0$ and $ a(x)/x\to \xi$ as $ x\to \infty$. An easy
computation shows as $ u\to \infty$,
\begin{align*}
	\frac{ \bar F(u+xa(u))}{\bar F(u) }  = & (1+o(1))\exp \left\{ - \int_{ u}^{ u+xa(u)} \frac{ 1}{a(t) }dt \right\}  \\
= & (1+o(1))\exp \left\{ - \int_{ u}^{ u+x\xi u(1+o(1))} \frac{ t}{a( t) } \frac{dt}{t} \right\}
\to  (1+\xi x)^{ -1/\xi}.
\end{align*}
This means that if $ X$ is a random variable having distribution $ F$
then for large $u$,
\[ 
 	P\Big[  \frac{ X-u}{a(u) }\le x\Big| X>u \Big] \approx  G_{ \xi,1}(x)   
\]
and $ a(u)$ is a choice for the scale parameter $
\beta (u)$ in \eqref{eq:pick:Bal:deH}. Hence we get  that $ \beta (u)/u\to \xi$ as $ u\to
\infty$ by the convergence to types theorem \citep{resnick:1987}.  

Consider the ME plot for iid random variables having common distribution $
F$ which satisfies $ \bar F\in RV_{ -1/\xi}$ for some $ \xi>0$.  Since the excess
distribution is well approximated by the GPD for high thresholds, we
expect that  for $ \xi<1$, the ME function will look similar
to that of the GPD for high thresholds and therefore seek evidence of
linearity in the plot. We first consider the ME plot when $ 0<\xi<1$
and will discuss the case $ \xi\ge1$ separately.  Furthermore, we see that for each $ n\ge 1$, the mean excess plot, being a finite set of $ \mathbb{R}^{ 2}$-valued random variables, is measurable and a random closed set. It follows from the definition of random sets; see Definition 1.1 in \cite[p. 2]{molchanov:2005}.

\subsubsection{Heavy tail with a finite mean; $0<\xi
  <1$.}\label{subsubsec:finitemean} 
The scaled and thresholded ME plot converges to a deterministic line.

\begin{theorem}\label{thm:positive:xi} 
If $ (X_{ n},n\ge 1)$ are iid observations with distribution $ F$ satisfying $ \bar F\in RV_{ -1/\xi }$ with $ 0<\xi<1$,  then in $ \mathcal{F}$
 \begin{equation}\label{eq:lim:pos} 
 	\mathcal{S}_{ n}:=\frac{ 1}{X_{ (k)}  }  \left\{   \big(X_{  (i)},\hat M(X_{ (i)})\big) :i=2,\ldots, k\right\}   \stackrel{ P}{\longrightarrow } \mathcal{S}:=\Big\{   \Big(t, \frac{ \xi}{1-\xi } t \Big):t\ge 1  \Big\} .
\end{equation}

\end{theorem}
 
 \begin{remark}\label{rem:normalisation} 
Roughly,  this result implies 
 \[ 
X_{(k)} \mathcal{S}_{ n}:=\left\{   \big(X_{ (i)},\hat M(X_{ (i)})\big) :i=2,\ldots, k_{ n} \right\}\approx X_{ (k)} \mathcal{S}.
\]

The plot of the points $ \mathcal{S}_{ n}$ is a little different from the original
ME plot. In practice,  people plot of the points $ \{   (X_{
  (i)}),\hat M(X_{ (i)}):1<i\le n  \}$ {but our result restricts
attention to the higher order statistics corresponding to 
$X_{(1)},\dots,X_{(k)}.$  This restriction is natural and corresponds
to looking at observations over  high thresholds. One imagines
 zooming into the area of interest in the complete ME plot. }

This result scales the points $ (X_{ (i)},\hat M(X_{ (i)}))$  by
$ X_{ (k)}^{ -1}$.  Since both  co-ordinates of the points in the
plot are scaled, we do not change the structure or appearance of the plot but only
the scale of the axes. Hence we may still  estimate the slope of the
line if we want to
estimate $ \xi$ by this method 
\citep{davison1990meo}. The scaling is important because the points $
\{   (X_{ (i)}),\hat M(X_{ (i)}):1<i\le k  \}$ are moving to infinity
and the Fell topology is not equipped to handle sets which are moving
out to infinity. Furthermore, the
{regular variation assumption}
 on the tail of $ F$ involves a ratio condition and thus it
is natural that the random set convergence uses scaling.

{A central assumption in Theorem \ref{thm:positive:xi} is that the random variables $ \{   X_{ i}  \}$ are iid.  The proof of the theorem below will explain that an important tool is the convergence of the tail empirical measure $ \hat{\nu}_{ n}$ in \eqref{eq:tailempirical}. By Proposition 2.1 in \cite{resnick1998tail}, we know that the iid assumption of the random variables is not a necessary condition for the convergence of the tail empirical measure. We believe as long as the tail empirical measure converges, our result should hold. }
\end{remark}
 
\begin{proof} 
 We show that for every subsequence $m_{ n} $ of integers there exists a further subsequence  $ l_{ n}$ of $ m_{ n}$ such that 
\begin{equation}\label{eq:almostsure} 
 	\mathcal{S}_{ l_{ n}}  \to \mathcal{S} \ \ \ \ \mbox{ a.s. }   
\end{equation}
Define the tail empirical measure as a random element of $ M_{ +}(0,\infty]$  by
\begin{equation}\label{eq:tailempirical} 
 	\hat\nu_{ n}:= \frac{ 1}{ k}\sum_{ i=1}^{ n} \delta _{ X_{ i}/X_{ (k)}}   .
\end{equation}
Following (4.21)  in  \cite[p.83]{resnick2006htp} we get that 
\begin{equation}\label{eq:measure.weak}
 	\hat\nu_{ n} \Rightarrow \nu    \ \ \ \ \mbox{ in } M_{ +}(0,\infty]
\end{equation}
where $ \nu(x,\infty]=x^{ -1/\xi}, x>0.$ Now consider 
\[ 
 	S_{ n}(u)= \Big( \frac{ X_{ (\lceil k u \rceil)}}{X_{ (k)} }, \frac{\hat M(X_{(\lceil ku\rceil) }) }{X_{ (k)}} \Big)    \ \ \ \  u\in (0,1]
\]
as random elements in $ D^{ 2}_{ l}(0,1]$.% the space of all 2-dimensional left-continuous functions with finite right limits equipped with the Skorokhod topology.
We will show that $S_{ n}(\cdot) \stackrel{ P}{ \longrightarrow }  S(\cdot) $ in $ D^{ 2}_{ l}(0,1]$, where
 \[ 
 	 S(u)=    \Big(u^{ -\xi}, \frac{ \xi}{1-\xi } u^{ -\xi} \Big) \ \ \ \ \mbox{ for all } 0<u\le 1.
\]
We already know the result for the first component of $ S_{ n}$, i.e., $ S^{ (1)}_{ n}(t):=X_{ (\lceil kt \rceil)}/X_{ (k)}\stackrel{ P}{\to  } t^{ -\xi} $ in $ D_{ l}(0,1]$; see \cite[p.82]{resnick2006htp}. Since the limits are non-random it suffices to prove the convergence of the second component of $ S_{ n}$.  Observe that the empirical mean excess function can be obtained from the tail empirical measure:
\[ 
 	S^{ (2)}_{ n}(u):= \frac{\hat M(X_{(\lceil ku \rceil)})}{X_{ (k)}}   = \frac{ k}{\lceil ku\rceil-1 }\int_{ X_{ (\lceil ku\rceil)}/X_{ (k)}}^{ \infty} \hat \nu_{ n}(x, \infty]dx.
\] 
Consider the maps $ T$ and $ T_{ K}$  from $M_{ +}(0,\infty]$ to $ D_{ l}[1,\infty)$ defined by
\[ 
 	T(\mu)(t) =  \int_{ t}^{ \infty}\mu (x,\infty] dx\ \ \mbox{ and }\ \ T_{ K}(\mu)(t) =  \int_{ t}^{ K\vee t}\mu (x,\infty] dx.   
\]
We understand $ T(\mu)(t)= \infty$ if $ \mu(x,\infty]$ is not integrable. We will show that $T(\hat \nu_{ n}) \stackrel{ P}{ \to}T(\nu)  $. The function $ T_{ K}$ is obviously continuous and therefore $ T_{ K}(\hat \nu_{ n})\stackrel{ P}{\to  } T_{ K}(\nu) $ in $ D_{ l}[1,\infty)$.
In order to prove that $T(\hat \nu_{ n}) \stackrel{ P}{ \to}T(\nu)  $ it suffices to show that for any $ \epsilon >0$
\[ 
 	\lim_{ K \rightarrow \infty } \limsup_{ n \rightarrow \infty  } P\Big[\big\| T_{ K}(\hat\nu_{ n}) -T(\hat\nu_{ n})\big\| >\epsilon  \Big]=0,    
\]
where $ \|\cdot\|$ is the supnorm on $ D_{ l}[1,\infty)$. To verify
this claim,  note that 
$$
\lim_{ K \rightarrow \infty } \limsup_{ n \rightarrow \infty  } P\Big[\big\| T_{ K}(\hat\nu_{ n}) -T(\hat\nu_{ n})\big\| >\epsilon  \Big]%\\
\leq \lim_{ K \rightarrow \infty } \limsup_{ n \rightarrow \infty  } P\Big[ \int_{ K}^{ \infty }\hat \nu_{ n}(x, \infty]dx>\epsilon  \Big] 
$$
and the rest is proved easily following the arguments used in the step 3 of the proof of Theorem 4.2 in \cite[p.81]{resnick2006htp}. 

Suppose $ D_{ l,\ge 1}(0,1]$ is the subspace of $ D_{ l}(0,1]$ consisting  only of functions which are never less than 1. Consider the random element $ Y_{ n}$ in the space $ D_{ l,\ge 1}(0,1]\times D_{ l}[1, \infty)$, 
%equipped with the product topology,
\[ 
 	Y_{ n}:=\Big( \frac{ X_{ (\lceil k \,\cdot \rceil)}}{X_{ (k)} }, T(\hat \nu_{ n}) \Big).    
\]
From what we have obtained so far it is easy to check that $ Y_{ n}\stackrel{ P}{\to }Y $, where 
\[ 
 	Y(u,t)=\Big(u^{ -\xi}, \frac{ \xi}{1-\xi } t^{ (\xi-1)/\xi}\Big)   .
\]
The map $ \tilde T:D_{ l,\ge 1}(0,1]\times D_{ l}[1, \infty)\to D_{ l}(0,1] $ defined by 
\[ 
 	\tilde T(f,g)(u)=g(f(u)) \ \ \ \ \mbox{ for all } 0<u\le 1   
\]
is continuous if $ g$ and $ f$ are continuous and therefore 
\[ 
 	\tilde T (Y_{ n})(u)=    \int\limits_{ X_{ (\lceil ku\rceil)}/X_{ (k)}}^{ \infty} \hat \nu_{ n}(x, \infty]dx \stackrel{ P}{\longrightarrow  }  \frac{ \xi}{1-\xi } u^{ 1-\xi} \ \ \ \ \mbox{ in } D_{ l}(0,1].
\]
This finally shows the convergence of the second component of $ S_{ n}$ and hence we get that $ S_{ n} \stackrel{ P}{ \to } S.$ 

Next we have to convert this result to that of convergence of the random sets $ \mathcal{S}_{ n} $. This argument is similar to the one used to prove  Lemma 2.1.3 in \cite{das2008qpr}. Choose any subsequence $ (m_{ n})$ of integers. Since $ S_{ n}(\cdot)\stackrel{ P}{\longrightarrow  }  S(\cdot)$ we have $ S_{ m_{ n}}(\cdot) \stackrel{ P}{\longrightarrow  }S(\cdot) $ in $ D^{ 2}_{ l}(0,1]$. So there exists a subsequence $ (l_{ n})$ of $( m_{ n})$ such that $ S_{ l_{ n}}(\cdot)\to S(\cdot) $ a.s. Now the final step is to use this to prove \eqref{eq:almostsure} and for that we will use Lemma \ref{lem:setconv} . Take any point in $ \mathcal{S}$ of the form $ (t,\xi/(1-\xi)t)$ for some $ t\ge 1$. Set $ u=t^{ -\xi}$ and observe that $ S_{ l_{ n}}(u)\to (t,\xi/(1-\xi)t)$ and $ S_{ l_{ n}}(u)\in \mathcal{S}_{ l_{ n}} $.  This proves condition $ (1)$ of Lemma \ref{lem:setconv} and we next prove condition $ (2)$. Suppose for some subsequence $ (j_{ n})$ of $ (l_{ n})$, $ S_{ j_{ n}}(u_{ n})$ converges to $ (x,y)$. Since $ S^{ (1)}(u)$ is strictly monotone we get that the must be some $ 0<u\le 1$ such that $ u_{ n}\to u$ as  $ n\to \infty$. Now, since $ S_{ j_{ n}}\to S$ and $ S$ is a continuous we get that $ S_{ j_{ n}}(u_{ n})\to S(u)\in \mathcal{S}$. That completes the proof.
\end{proof}

\subsubsection{Case  $  \xi\geq 1$; limit sets are random.}
The following theorem describes the asymptotic behavior of 
the ME plot  when $ \xi\ge1$. Reminder: $\xi>1$  guarantees an infinite mean.

\begin{theorem}\label{thm:xi:pos:gtr1} 
Assume $ (X_{ n},n\ge 1)$ are i.i.d. observations with distribution $ F$ satisfying $ \bar F\in RV_{ -1/\xi }$:
\begin{enumerate}[(i)]
\item If  $ \xi>1$, then 
 \begin{equation}\label{eq:lim:pos:gtr1} 
 	\mathcal{S}_{ n}:= \left\{   \left(\frac{X_{ (i)}}{b(n/k)},\frac{\hat M(X_{ (i)})}{b(n)/k}\right) :i=2,\ldots, k \right\}   \Longrightarrow \mathcal{S}:=\Big\{   \Big(t^{ \xi},  tS_{ 1/\xi} \Big):t\ge 1  \Big\} 
\end{equation}
in $ \mathcal{F}$, where $ b(n):= F^{ \leftarrow }\big( 1-1/n \big) $ and $ S_{ 1/\xi}$ is the positive  stable random variable with index $ 1/\xi$ which satisfies for $ t\in \mathbb{R}$
\begin{equation}\label{eq:stable} 
 	E\big[  e^{ itS_{ 1/\xi}} \big]   =
	\exp \Big\{ -\Gamma \Big( 1- \frac{ 1}{\xi } \Big)\cos  \frac{ \pi}{2\xi }\big| t \big|^{ 1/\xi}\Big[  1-i\sgn(t) \tan \frac{ \pi}{2\xi } \Big]   \Big\} 
\end{equation}
\item If $ \xi=1$  and $ k$ satisfies $k=k(n)\to\infty$, $k/n
  \to 0$, and 
\begin{equation}\label{eqn:beesGoCrazy}
kb(n/k)/b(n)\to 1\quad (n\to\infty),
\end{equation} then 
 \begin{align}
 	\mathcal{S}_{ n}:=& \left\{   \left(\frac{X_{
                (i)}}{b(n/k)},\frac{\hat M(X_{ (i)})}{b(n/k)}- \frac{k
              C_{ n,k}}{ib(n) }\right) :i=2,\ldots, k
        \right\}\nonumber\\
      &  \Longrightarrow \mathcal{S}:=\Big\{   t\Big(1,  S_{ 1}-1-\log
        t \Big):t\ge 1  \Big\}  \label{eq:lim:pos:eql1} 
\end{align}
in $ \mathcal{F}$, where 
\[ 
 	C_{ n,k} = n\big( E[X_{ 1}I_{ X_{ 1}\le b(n)]}]- E[X_{ 1}I_{ X_{ 1}\le b(n/k)]}]\big)
\]
and $ S_{ 1}$ is a positively skewed stable random variable satisfying
\[ 
 	E\big[  e^{ itS_{ 1}} \big]     =\exp \Big\{  i t\int_{ 0}^{ \infty} \Big( \frac{ \sin x}{ x^{ 2}}- \frac{ 1}{x(1+x) } \Big)dx -  | t | \Big[  \frac{\pi}{2}+i\sgn(t) \log | t |  \Big] \Big\}.
\]
\end{enumerate}

\end{theorem}

\begin{remark}\label{rem:xi:gtr1} 
In Theorem \ref{thm:positive:xi} we considered the points of the mean
excess plot normalized by $ X_{ (k)}$. By scaling both coordinates by
the same normalizing sequence, we did not change the structure of the
plot. But in Theorem \ref{thm:xi:pos:gtr1}\emph{(i)}  we need
different scaling in the two coordinates.  This is simple to observe
since $ b(n)=F^{ \leftarrow }(1-n^{ -1})\in RV_{ \xi}$ and $ \xi>1$
implies $ kb(n/k)/b(n)\to 0$ as $ n\to \infty$. This means that in
order to get a finite limit we need to normalize the second coordinate
by a sequence increasing at a much faster rate than the normalizing
sequence for the first coordinate. This is indeed changing the
structure of the plot and even with this normalization the
limiting set is random.
 The limit is a curve scaled in the
second coordinate by the random quantity $ S_{ 1/\xi}$. 
Note that the limit is independent of the choice of the
sequence $ k_{ n}$ as long as it satisfies the condition that $ k_{
  n}\to \infty$ and $ k_{ n}/n\to 0$ as $ n\to \infty$. Another interesting outcome, as pointed out by a referee, is that in the log-log scale the limit set becomes
  \[ 
 	\log \mathcal{S}= \{  ( u, \frac{ 1}{\xi } u+ \log S_{ 1/\xi}):u\ge 0 \}   
\]
which is a straight line with slope $ 1/\xi$ and a random intercept term $ S_{ 1/\xi}$.

In Theorem \ref{thm:xi:pos:gtr1}\emph{(ii)} along with $ \xi=1$ we
make the extra assumption \eqref{eqn:beesGoCrazy}.  Under these
assumptions we get that mean excess plot with some centering in the
second coordinate converges to a random set. {We could obtain
result without \eqref{eqn:beesGoCrazy} but then the centering becomes
random and more complicated and difficult to interpret. The
significance of \eqref{eqn:beesGoCrazy} is as follows: The centering
$C_{n,k}$ is of the form
$$C_{n,k}=n\bigl(\pi (n)-\pi(n/k)\bigr)$$ where
$\pi(t)=\int_0^{b(t)} \bar F(s) ds$ is in the de Haan class $\Pi$ and
has slowly varying auxiliary function $g(t):=b(t)/t$; see 
\cite{resnick2006htp}, \cite{bingham:goldie:teugels:1989}, \cite{DeHaan:1976p6219}  and \cite{dehaan:ferreira:2006}.  Condition
\eqref{eqn:beesGoCrazy}  is the same as requiring $k$ to satisfy
$g(n/k)/g(n) \to 1.$}
\end{remark}

\begin{proof}
\emph{(i)} We will first  prove that 
\begin{equation}\label{eq:conv:func}
	 Y_{ n}(t):= \Bigg( \frac{ X_{( \lceil k/t \rceil)}}{ b(n/k)}, \frac{ \hat M\big( X_{( \lceil k/t \rceil)} \big) }{b(n)/k } \Bigg)  \Longrightarrow Y(t):=\big( t^{ \xi},tS_{ 1/\xi} \big) \ \ \ \ \mbox{ in } D^{ 2}[1,\infty).
\end{equation}
The two important facts that we will need for the proof are the following:
\begin{enumerate}[(A)] 
        \item  \cite{csorgo1986ads}  showed that for any $ k_{ n}\to \infty$ satisfying $ k_{ n}/n \to 0$ 
\[
 	\frac{ 1}{b(n) }\sum_{ i=1}^{ k_{ n}} X_{ (i)} \Longrightarrow S_{ 1/\xi}, \ \ \ \ \mbox{ in }\mathbb{R}.
\]
\item Under the same assumption on the sequence $ k_{ n}$
  \citep[p.82]{resnick2006htp}  
\begin{equation}\label{eq:xieql1:orderstat} 
 	Y^{ (1)}_{ n}   (t)= \frac{ X_{( \lceil k/t \rceil)}}{ b(n/k)} \stackrel{ P}{ \longrightarrow  } Y^{ (1)}(t)= t^{ \xi} \ \ \ \ \mbox{ in } D[1, \infty).
\end{equation}
\end{enumerate} 

Since $ Y^{ (1)}(t)$ is non-random,  in order to prove
\eqref{eq:conv:func} it suffices to show that $ Y_{ n}^{
  (2)}(t)\Longrightarrow Y^{ (2)}(t)$ in $ D[1,\infty)$  
\citep[Proposition 3.1, p.57]{resnick2006htp}. By Theorem 16.7
in \cite[p.174]{billingsley1999cpm}  we need to show that $ Y_{ n}^{
  (2)}(t)\Longrightarrow Y^{ (2)}(t)$ in $ D[1,N]$ for every $ N> 1$. So
fix $ N>1$ arbitrarily. By an abuse of notation we will use $ Y$ and $
Y_{ n}$ as to denote their restrictions on $ [1,N]$ as elements of $
D[1,N]$.  
 
 Observe that $ b(n)\in RV_{ \xi}$ and since $ \xi>1$ we get $ kb(n/k)/b(n)\to 0$ as $ n\to \infty$. Combining this  with (B)  we get that for any $ t\ge 1$,
\begin{equation}\label{eq:one:dim}
	\frac{ kX_{ (\lceil k/t \rceil)}}{ b(n) } \stackrel{ P}{ \longrightarrow  } 0. 
\end{equation}
Also observe that for any $ 1\le t_{ 1}<t_{ 2}\le N$
\begin{equation}\label{eq:lim:diff}
	\frac{ 1}{ b_{ n}} \sum_{ i=\lceil k/t_{ 2} \rceil+1}^{ \lceil k/t_{ 1} \rceil}X_{ (i)}\le k \left(\frac{ 1}{t_{ 1} }- \frac{ 1}{t_{ 2} }+1 \right) \frac{ X_{ (\lceil k/t_{ 2} \rceil)}}{b(n) } \stackrel{ P}{\longrightarrow  } 0. 
\end{equation}

Using (A), \eqref{eq:one:dim}, \eqref{eq:lim:diff} and Proposition 3.1 in \cite[p.57]{resnick2006htp} we get that for any $ 1\le t_{ 1}<t_{ 2}\le N$
\begin{equation}\label{eq:two:dim:skeleton}
	\frac{ 1}{b(n) }\Big( \sum_{ i=1}^{ \lceil k/t_{ 2} \rceil-1}X_{ (i)},\sum_{ i=\lceil k/t_{ 2} \rceil}^{ \lceil k/t_{ 1} \rceil-1}X_{ (i)},k X_{ (\lceil k/t_{ 2} \rceil)},kX_{ (\lceil k/t_{ 1} \rceil)} \Big) \Longrightarrow  \big( S_{ 1/\xi},0,0,0 \big) .
\end{equation}
This allows us to obtain the weak limit of $ (Y_{ n}^{ (2)}(t_{ 1}),Y_{ n}^{ (2)}(t_{ 2}))$:
\begin{align*}
&(Y_{ n}^{ (2)}(t_{ 1}),Y_{ n}^{ (2)}(t_{ 2})) \\
&=  \frac{ k}{ b(n)} \Big( \hat M \big( X_{ (\lceil k/t_{ 1} \rceil)} \big), \hat M \big( X_{ (\lceil k/t_{ 2} \rceil)} \big) \Big) \\
&= \frac{ k}{ b(n)} \Bigg( \frac{ 1}{\lceil k/t_{ 1} \rceil -1 } \sum_{ i=1}^{ \lceil k/t_{ 1} \rceil -1} X_{ (i)}-X_{ (\lceil k/t_{ 1} \rceil)},\frac{ 1}{\lceil k/t_{ 2} \rceil -1 } \sum_{ i=1}^{ \lceil k/t_{ 2} \rceil -1} X_{ (i)}-X_{ (\lceil k/t_{ 2} \rceil)}        \Bigg)\\
&= \frac{1}{b(n)} \Bigg( \frac{ k  \sum\limits_{ i=1}^{ \lceil k/t_{ 2} \rceil -1} X_{ (i)} }{\lceil k/t_{ 1} \rceil -1 }     +    \frac{ k  \sum\limits_{ i= \lceil k/t_{ 2} \rceil}^{ \lceil k/t_{ 1} \rceil -1} X_{ (i)} }{\lceil k/t_{ 1} \rceil -1 }   -k X_{ (\lceil k/t_{ 1} \rceil)}  ,   \frac{ k  \sum\limits_{ i=1}^{ \lceil k/t_{ 2} \rceil -1} X_{ (i)} }{\lceil k/t_{ 2} \rceil -1 } - k X_{ (\lceil k/t_{ 2} \rceil)}  \Bigg)\\
&\Longrightarrow  (t_{ 1},t_{ 2})S_{ 1/\xi}.
\end{align*}
By similar arguments we can also show that for any $ 1\le t_{ 1}<t_{ 2}<\cdots <t_{ m}\le N$, 
\[
	 (Y_{ n}^{ (2)}(t_{ 1}),\ldots,Y^{ (2)}_{ n}(t_{ m}))\Longrightarrow (t_{ 1},\cdots,t_{ m})S_{ 1/\xi}.
\]
From  \cite{billingsley1999cpm}, Theorem 13.3, p.141, the proof of
\eqref{eq:conv:func} will be complete if we show  for any $
\epsilon >0$ 
\[ 
 	\lim_{ \delta  \rightarrow 0  } \lim_{ n \rightarrow \infty  } P \big[  w_{ N}(Y^{ (2)}_{ n},\delta )\ge \epsilon  \big]  =0, 
\]
where for any $ g\in D[1,N]$
\[ 
 	   w_{ N}(g,\delta )= \sup_{ 1\le t_{ 1}\le t\le t_{ 2}\le N,t_{ 2}-t_{ 1}\le \delta }\big\{ \big| g(t)-g(t_{ 1}) \big| \wedge \big| g(t_{ 2})-g(t) \big|   \big\} .
\]
Fix any $ \epsilon >0$ and choose $ n$ large enough such that $ X_{ (k)}>0$ and $ k/N>1$. Then for any $ 1\le t_{ 1}\le t_{ 2}\le N$
\begin{align*}
& \Big| Y_{ n}^{ (2)}(t_{ 2})-Y_{ n}^{ (2)}(t_{ 1}) \Big| \\
& = \frac{ 1}{b(n) }\Bigg|  \frac{ k}{ \lceil k/t_{ 2} \rceil-1 }\sum_{ i=1}^{ \lceil k/t_{ 2} \rceil-1}X_{ (i)} -kX_{ (\lceil k/t_{ 2} \rceil)}-\frac{ k}{ \lceil k/t_{ 1} \rceil-1 }\sum_{ i=1}^{ \lceil k/t_{ 1} \rceil-1}X_{ (i)} +kX_{ (\lceil k/t_{ 1} \rceil)}\Bigg|\\
& \le \frac{ 1}{ b(n)}   \Bigg(\frac{ k}{ \lceil k/t_{ 2} \rceil-1 }-\frac{ k}{ \lceil k/t_{ 1} \rceil-1 }  \Bigg)  \sum_{ i=1}^{ \lceil k/N \rceil-1} X_{ (i)}\\
& \ \ \ \ + \frac{ 1}{ b(n)} \Bigg(  \frac{ k}{ \lceil k/t_{ 2} \rceil-1 }\sum_{ i=\lceil k/N \rceil}^{ \lceil k/t_{ 2} \rceil-1}X_{ (i)} +kX_{ (\lceil k/t_{ 2} \rceil)} + \frac{ k}{ \lceil k/t_{ 1} \rceil-1 }\sum_{ i=\lceil k/N \rceil}^{ \lceil k/t_{ 1} \rceil-1}X_{ (i)} +kX_{ (\lceil k/t_{ 1} \rceil)}  \Bigg)\\
& \le  \frac{ 1}{ b(n)}   \Bigg(\frac{ k}{ \lceil k/t_{ 2} \rceil-1 }-\frac{ k}{ \lceil k/t_{ 1} \rceil-1 }  \Bigg)  \sum_{ i=1}^{ \lceil k/N \rceil-1} X_{ (i)} + \frac{ 4kN}{b(n) } X_{ (\lceil k/N \rceil)}\\
&=: U_{ n,N}(t_{ 1},t_{ 2}) \Longrightarrow (t_{ 2}-t_{ 1}) S_{ 1/\xi}.
\end{align*}
%Note the random variables two lines previously do not depend on $t_1$ or $t_2.$ 
Therefore, {using the form of the function $ U_{ n,N}$ we get}
\begin{align*}
& \lim_{ \delta  \rightarrow 0  } \lim_{ n \rightarrow \infty  } P \big[  w_{ N}(Y_{ n}^{ (2)},\delta )\ge \epsilon  \big]\\
& \le  \lim_{ \delta  \rightarrow 0  } \lim_{ n \rightarrow \infty  } P \Big[  \sup_{ 1\le t_{ 1}\le t_{ 2}\le N,t_{ 2}-t_{ 1}\le \delta } \Big| Y_{ n}^{ (2)}(t_{ 2})-Y_{ n}^{ (2)}(t_{ 1}) \Big|\ge \epsilon   \Big] \\
& \le \lim_{ \delta  \rightarrow 0  } \lim_{ n \rightarrow \infty  } P \Big[    \sup_{ 1\le t_{ 1}\le t_{ 2}\le N,t_{ 2}-t_{ 1}\le \delta } U_{ n,N}(t_{ 1},t_{ 2})  >\epsilon \Big]\\
&= \lim_{ \delta  \rightarrow 0  } P\Big[  \delta S_{ 1/\xi}\ge \epsilon  \Big]=0. 
\end{align*}
Hence we have proved \eqref{eq:conv:func}.

Now we prove the statement of the theorem. By Proposition 6.10, page 87 in
\cite{molchanov:2005} it suffices to show that for any
continuous function  $ f:\mathbb{R}^{ 2}\mapsto \mathbb{R}_{ +}$ with
a compact support 
\[ 
 	\lim_{ n \rightarrow \infty  }E\Big[  \sup_{ x\in \mathcal{S}_{ n}} f(x) \Big]    = E\Big[  \sup_{ x\in \mathcal{S}} f(x) \Big]  .
\]
Suppose $ f:\mathbb{R}^{ 2}\mapsto \mathbb{R}_{ +}$ is a continuous function with compact support. By the Skorokhod representation theorem (see Theorem 6.7  in \cite[p.70]{billingsley1999cpm}) there exists a probability space $ (\Omega ,\mathcal{G},P)$ and random elements $ Y_{ n}^{ *}(t)$ and $ Y^{ *}(t)$ in $ D[1, \infty)$ such that $ Y_{ n}\stackrel{ d}{= }Y_{ n}^{ *} $ and $ Y\stackrel{ d}{= } Y^{ *} $ and $ Y_{ n}^{ *}(t)(\omega ) \to Y^{* }(t)(\omega ) $ in $ D[1,\infty)$ for every $ \omega \in \Omega $. Now observe that 
\[ 
 	\sup_{ x\in \mathcal{S}}f(x) \stackrel{ d}{= } \sup_{ t\ge 1} f\big(Y^{ *}(t)\big) \ \ \ \ \mbox{ and }    \ \ \ \  \sup_{ x\in \mathcal{S}_{ n}}f(x) \stackrel{ d}{= } \sup_{ t\ge 1} f\big(Y_{ n}^{ *}(t)\big).
\] 
Since $ f$ is continuous we get 
\[
	 \sup_{ t\ge 1} f\big(Y_{ n}^{ *}(t)\big) \to \sup_{ t\ge 1} f\big(Y^{ *}(t)\big)\ \ \ \ P-a.s.
\] 
and since $ f$ is bounded we apply the dominated convergence theorem to get 
\[ 
 	   \lim_{ n \rightarrow \infty  }E\Big[  \sup_{ x\in \mathcal{S}_{ n}} f(x) \Big]= \lim_{ n \rightarrow \infty}E \Big[  \sup_{ t\ge 1} f\big(Y_{ n}^{ *}(t)\big)  \Big]   = E \Big[  \sup_{ t\ge 1} f\big(Y^{ *}(t)\big)  \Big] =E\Big[  \sup_{ x\in \mathcal{S}} f(x) \Big]
\]
and that completes the proof of the theorem when $ \xi>1$.

\emph{(ii)} Similar to the proof of part \emph{(i)} we will first
prove that 
in  $D^{ 2}[1,\infty),$
\begin{equation}\label{eq:conv:func:xieql1}
	 Y_{ n}(t):= \Bigg( \frac{ X_{( \lceil k_{ n}/t \rceil)}}{
           b(n/k)}, \frac{ \hat M\big( X_{( \lceil k_{ n}/t \rceil)}
           \big) }{b(n/k) }- \frac{ k}{\lceil k/t\rceil }\frac{ C_{ n,k}}{b(n) } \Bigg)  \Longrightarrow Y(t):=t\big(1 ,S_{ 1}-1-\log t \big) 
\end{equation}
We will use the following facts:
\begin{enumerate}[(A)] 
        \item  \cite{csorgo1986ads}  showed that for any $ k_{ n}\to \infty$ satisfying $ k_{ n}/n \to 0$ 
\[
 	\frac{ 1}{b(n) }\Bigg(\sum_{ i=1}^{ k_{ n}} X_{ (i)}-C_{ n,k}\Bigg) \Longrightarrow S_{1}, \ \ \ \ \mbox{ in }\mathbb{R}.
\]
\item For $k\to \infty$ with $k/n\to 0$, \eqref{eq:xieql1:orderstat}
  still holds with $\xi=1$.
\end{enumerate} 
By the same arguments used in part \emph(i) it suffices to prove that for any arbitrary $ N> 1$
\[ 
 	Y^{ (2)}_{ n}(t) \Longrightarrow Y^{ (2)}(t) \ \ \ \ \mbox{ in } D[1,N]  .
\]
Observe that from \eqref{eq:xieql1:orderstat}  and the assumption that $ kb(n/k)/b(n)\to 1$  we  get for any $ t>1$
\begin{equation}\label{eq:xieql1:midpart}
	\frac{ 1}{ b(n)}\sum_{i= \lceil k/t\rceil }^{ k}X_{ (i)} = \frac{(1+o(1))
          }{b(n/k) k } \sum_{ i=\lceil k/t\rceil}^{ k}X_{ (i)}
        \stackrel{ P}{\longrightarrow  } \log t.  
\end{equation}
The reason for this is that
$$
\frac 1k \sum_{ i=\lceil k/t \rceil}^{ k} \frac{X_{ (i)}}{b(n/k)}
=\int_{X_{(k)}/b(n/k)}^{ X_{(\lceil k/t \rceil)}/b(n/k)} x\nu_n(dx)
\stackrel{P}{\to} \int_1^t x x^{-2} dx =\log t
$$
where $\nu_n (dx) =\frac 1k \sum_{i=1}^n \delta_{X_{(i)}/b(n/k)}
(dx) \to x^{-2}dx.$ 
See \eqref{eq:measure.weak} 
and  \eqref{eq:xieql1:orderstat}.
Now fix any $1\le t\le N  $ and note that 
%\marginnote{Fixing t is not
%  the same thing as $D[1,\infty)$ convergence but maybe wait to see if
%  the referee objects.}
\begin{align*}
&Y_{ n}^{ (2)}(t) =   \frac{\hat M \big( X_{ (\lceil k/t \rceil)}
  \big)}{b(n/k)}- \frac{kC_{ n,k}}{\lceil k/t\rceil b(n)} \\
&= \frac{ 1}{ kb(n/k)} \Bigg( \frac{ k}{\lceil k/t \rceil -1 } \sum_{
  i=1}^{ \lceil k/t \rceil -1} X_{ (i)}\Bigg)-\frac{X_{ (\lceil k/t
    \rceil)}}{b(n/k)}-  \frac{kC_{ n,k}}{\lceil k/t\rceil b(n)}\\
& =(1+o_p(1))\frac{ t}{ b(n)} \Bigg( \sum_{ i=1}^{ k} X_{ i} -C_{ n,k}
\Bigg) -  \frac{X_{ (\lceil k/t \rceil)}}{b(n/k)} -  \frac{ (1+o(1))t}{b(n) }
\sum_{ i=\lceil k/t \rceil}^{ k} X_{ (i)}\\
&\Longrightarrow tS_{ 1} -t-t\log t.
\end{align*}
We complete the proof using the same arguments as those in part \emph{(i)}.
\end{proof}

\end{subsection}

\begin{subsection}{Negative Slope}\label{sec:neg:xi}

The case when $ \xi<0$ is characterized by the following theorem which is a combination of Theorems 3.3.12 and 3.4.13(b) in \cite{embrechtskluppelbergmikosch:1997}:
\begin{theorem}\label{thm:char:neg:xi}
If $ \xi<0$ then the following are equivalent for a distribution function $ F$:
\begin{enumerate} 
        \item  $ F$ has a finite right end point $ x_{ F}$ and $ \bar F(x_{ F}-x^{ -1})\in RV_{ 1/\xi}$.
        \item $ F$ is in the maximal domain of attraction of a Weibull distribution with parameter $ -1/\xi$, i.e.,
        \[ 
 	   F^{ n}\big(  x_{ F}-c_{ n}x \big)\to \exp\{   -(-x)^{ -1/\xi}  \} \ \ \ \ \mbox{ for all } x\le 0,
\]
where $ c_{ n}=x_{ F}-F^{ \leftarrow }(1-n^{ -1})$.
\item There exists a measurable function $ \beta (u)$ such that 
\[ 
 	\lim_{ u \rightarrow x_{ F}  } \sup_{ u\le x\le x_{ F}} \big| F_{ u}(x)-G_{ \xi,\beta (u)}(x) \big|=0.   
\]
\end{enumerate} 
\end{theorem}

 Here we again get a characterization of this class of distributions
 in terms of the behavior of the maxima of iid  random variables and
 the excess distribution. Using Theorem \ref{thm:char:neg:xi}(1) and
 Karamata's Theorem 
 \citep[Theorem 1.5.11, p.28]{bingham:goldie:teugels:1989} we get that $ M(u)/(x_{
   F}-u)\sim \xi/(\xi-1)$ as $ u\to x_{ F}$.  We show that this
 behavior is observed empirically. The Pickands-Balkema-de
 Haan Theorem, part (3) of Theorem \ref{thm:char:neg:xi}, does
 not explicitly construct  the scale parameter $ \beta (u)$ but
 as in Remark \ref{rem:normalisation} one
 can show that $ \beta (u)/(x_{ F}-u)\to -\xi$ as $ u \to x_{ F}$.  
 
\begin{theorem}\label{thm:negative:xi} 
 Suppose $ (X_{ n},n\ge 1)$ are iid random variables with distribution $ F$ which has a finite right end point $ x_{ F}$ and satisfies $ 1-F(x_{ F}-x^{ -1})\in RV_{ 1/\xi}$ as $ x\to \infty$ for some $ \xi<0$. Then 
 \begin{align*} 
 	\mathcal{S}_{ n}&:= \frac{ 1}{X_{ (1) } -X_{ (k)}}\left\{ \Big(X_{ (i)}-X_{ (k)} , \hat M(X_{ (i)})  \Big):1<i\le k \right\} \\
	&\stackrel{ P}{\longrightarrow  }   \mathcal{S}:=\Big\{  \Big( t, \frac{ \xi}{1-\xi }(t-1)  \Big):0\le t\le 1    \Big\}
\end{align*}
in $ \mathcal{F}$.
\end{theorem}

\begin{remark}\label{rem:norm:negxi} 
As in  Subsection \ref{subsec:posxi} we look at a modified version of the mean excess plot. Here we scale and relocate the points of the plot near the right end point.  We may interpret this result as 
\[ 
 	\{   (X_{ (i)},\hat M(X_{ (i)})):1<i\le k  \}   \approx \big(X_{ (k)},0\big)+\big(X_{ (1)}-X_{ (k)}  \big)\mathcal{S} 
\]
where $ \mathcal{S}= \Big\{  \Big( t, \frac{ \xi}{1-\xi }(t-1)  \Big):0\le t\le 1    \Big\}$.
\end{remark}

\begin{proof} 
 The proof is similar to  that of Theorem \ref{thm:positive:xi}. From
Theorem 5.3(ii), p. 139 in \cite{resnick2006htp} we get 
 \[ 
 	\nu_{ n}:= \frac{ 1}{ k}\sum_{ i=1}^{ n} \delta_{ \frac{ x_{ F}-X_{ i}}{c_{ \lceil n/k\rceil} }}    \Longrightarrow \nu \ \ \ \ \mbox{ in } M_{ +}[0,\infty)
\]
 where $ \nu[0,x)=x^{ -1/\xi}$ for all $ x\ge 0$ and 
$c_n=F^\leftarrow (1-n^{-1}).$  Following the arguments used in the proof of Theorem 4.2 in \cite[p.81]{resnick2006htp}  we also get 
 \begin{equation}\label{eq:1}
 	\hat\nu_{ n}:= \frac{ 1}{ k}\sum_{ i=1}^{ n} \delta_{ \frac{ x_{ F}-X_{ i}}{x_{ F}-X_{ (k)} }}    \Longrightarrow \nu \ \ \ \ \mbox{ in } M_{ +}[0,\infty).
\end{equation}
Here we can represent $ \hat M(X_{( \lceil ku \rceil)})$ in terms of the empirical measure as
\[ 
 	\hat M(X_{( \lceil ku \rceil)})= \frac{ k(x_{ F}-X_{ (k)})}{\lceil ku\rceil-1 } \int_{ 0}^{ \frac{ x_{ F}-X_{ (\lceil ku\rceil)}}{ x_{ F}-X_{ (k)}}} \hat \nu_{ n} [0,x)dx   
\]
and taking the same route as in Theorem \ref{thm:positive:xi} we get 
\[ 
 	S_{ n}(u)= \Big( \frac{ x_{ F}-X_{ (\lceil k u \rceil)}}{x_{ F}-X_{ (k)} }, \frac{\hat M(X_{(\lceil ku\rceil) }) }{x_{ F}-X_{ (k)}} \Big)   \stackrel{ P}{\longrightarrow }  S(u)=    \Big(u^{ -\xi}, \frac{ \xi}{\xi-1 } u^{ -\xi} \Big)
\]
 in $ D^{ 2}_{ l}(0,1]$. From this we get in the Fell topology
 \[ 
 	\Big\{ \Big( \frac{ x_{ F}-X_{ (i)}}{ x_{ F}-X_{ (k)}}, \frac{ \hat M(X_{ (i)})}{x_{ F}-X_{ (k)} } \Big) :1\le i\le k   \Big \}   \stackrel{ P}{ \longrightarrow} \Big\{ \Big( t, \frac{ \xi}{\xi-1 }t \Big):0\le t\le 1  \Big\} 
\]
Finally, using the fact that 
\[ 
 	\frac{ X_{ (1)}-X_{ (k)}}{  x_{ F}-X_{ (k)}  }\stackrel{ P}{\longrightarrow  }1,    
\]
and the identity
\begin{align*} 
 	\frac{ 1}{X_{ (1) } -X_{ (k)}}&\left\{ \Big(X_{ (i)}-X_{ (k)} , \hat M(X_{ (i)})  \Big):1<i\le k \right\}\\
	= &\frac{ x_{ F}-X_{ (k)}}{X_{ (1) } -X_{ (k)}}\left\{ \Big(\frac{X_{ (i)}-X_{ (k)}}{x_{ F}-X_{ (k)}} , \frac{\hat M(X_{ (i)})}{x_{ F}-X_{ (k)}}  \Big):1<i\le k \right\}   \\
	 =&   \frac{ x_{ F}-X_{ (k)}}{X_{ (1) } -X_{ (k)}}\left\{ \Big(1-\frac{x_{ F}-X_{ (i)}}{x_{ F}-X_{ (k)}} , \frac{\hat M(X_{ (i)})}{x_{ F}-X_{ (k)}}  \Big):1<i\le k \right\}
\end{align*}
we get the final result.
\end{proof}

\end{subsection}

\begin{subsection}{Zero Slope}\label{sec:xi:zero}
The next result follows from Theorems 3.3.26 and 3.4.13(b) in \cite{embrechtskluppelbergmikosch:1997} and Proposition 1.4 in \cite{resnick:1987}.

\begin{theorem}\label{thm:char:zero:xi}
The following conditions are equivalent for a distribution function $ F$ with right end point $ x_{ F}\le \infty$: 
 \begin{enumerate} 
        \item  There exists  $ z<x_{ F}$ such that $ F$ has a representation 
        \begin{equation}\label{eq:char:F:gumbel} 
 	\bar F(x)=c(x)\exp \Big\{ -\int_{ z}^{ x} \frac{ 1}{a(t) }dt \Big\} , \ \ \ \ \mbox{ for all } z<x<x_{ F},   
\end{equation}
where $ c(x)$ is a  measurable function satisfying $ c(x)\to c>0$,  $ x\rightarrow x_{ F}$, and $ a(x)$ is a positive, absolutely continuous function with density $ a^{ \prime}(x)\to 0$ as $ x\to x_{ F}$.
\item $ F$ is in the maximal domain of attraction of the Gumbel distribution, i.e., 
\[ 
 	F^{ n}\big( c_{ n}x +d_{ n}\big)\to \exp \big\{-e^{ -x}  \big\}   \ \ \ \ \mbox{ for all }  x\in \mathbb{R},
\]
where $d_{ n}=F^{ \leftarrow }(1-n^{ -1}) $ and $ c_{ n}=a(d_{ n})$.
\item There exists a measurable function $ \beta (u)$ such that 
\[ 
 	\lim_{ u \rightarrow x_{ F}  }\sup_{ u\le x\le x_{ F}} \big| F_{ u}(x)-G_{ 0,\beta (u)}(x) \big|    =0.
\]
\end{enumerate}
\end{theorem}
Theorem 3.3.26 in \cite{embrechtskluppelbergmikosch:1997} also says that a possible choice of the auxiliary function $ a(x)$ in \eqref{eq:char:F:gumbel} is
\[ 
 	a(x)   =\int_{ x}^{ x_{ F}} \frac{ \bar F(t)}{ \bar F(x) } dt \ \ \ \ \mbox{ for all }x<x_{ F},
\]
and for this choice, the auxiliary function is  the ME function, i.e., $ a(x)=M(x)$. Furthermore, we also know that $ a^{ \prime }(x)\to 0$ as $ x\to x_{ F}$ and this  implies that $ M(u)/u\to 0$ as $ u\to x_{ F}$. 

A prime example in this class is the exponential distribution for which the ME function is a constant. The domain of attraction of the Gumbel distribution is  a very big class including distributions as diverse as the normal and the log-normal. It is indexed by  auxiliary functions  which only  need to satisfy $ a^{ \prime }(x)\to 0$ as $ x\to x_{ F}$. Since $ M(x)$ is a choice for the auxiliary function $ a(x)$, the class of ME functions corresponding to the domain of attraction of the Gumbel distribution is very large. 

\begin{theorem}\label{thm:zero:xi} 
  Suppose $ (X_{ n},n\ge 1)$ are iid random variables with distribution $ F$ which satisfies any one of the conditions in Theorem \ref{thm:char:zero:xi}. Then in $ \mathcal{F}$,
 \[ 
 	\mathcal{S}_{ n}:=\frac{ 1}{X_{ (\lceil k/2 \rceil)}-X_{ (k)} }\left\{\Big( X_{ (i)}-X_{ (k)}  , \hat M(X_{ (i)})  \Big):1<i\le k\right\} \stackrel{ P}{\longrightarrow  } \mathcal{S}:=\Big\{  \big( t, 1 \big):t\ge 0    \Big\}.
\]

\end{theorem}

\begin{proof} 
This is again similar to the proof of Theorem  \ref{thm:positive:xi}. Using Theorem \ref{thm:char:zero:xi}(2) we get 
\[ 
 	n    \bar F(c_{ n} x+d_{ n})\to e^{ -x} \ \ \ \ \mbox{ for all }x\in \mathbb{R}.
\]
Since $ n/k_{ n}\to \infty$ we also get 
\[ 
 	   \frac{ n}{k } \bar F \big( c_{ \lceil n/k \rceil}x+d_{ \lceil n/k \rceil} \big) \to e^{ -x} \ \ \ \ \mbox{ for all }x\in \mathbb{R}
\]
and then Theorem 5.3(ii) in \cite[p.139]{resnick2006htp} gives us
\[ 
 	\nu_{ n}:= \frac{ 1}{ k} \sum _{ i=1}^{ n}   \delta _{ \frac{ X_{ i}-d_{ \lceil n/k \rceil}}{ c_{ \lceil n/k \rceil}}} \Longrightarrow \nu \ \ \ \  \mbox{ in } M_{ +}(\mathbb{R})
\]
where $ \nu(x,\infty)=e^{ -x}$ for all $ x\in \mathbb{R}$. Following the arguments in the proof of Theorem 4.2 in \cite[p.81]{resnick2006htp} we get 
\[ 
 	\frac{ X_{ (k)}-d_{ \lceil n/k \rceil}}{ c_{ \lceil n/k \rceil}}\stackrel{ P}{\longrightarrow  }  0   
\]
and then
\[ 
 	\hat \nu_{ n} := \frac{ 1}{k } \sum_{ i=1}^{ n} \delta _{ \frac{ X_{ i}-X_{ (k)}}{ c_{ \lceil n/k \rceil}}} \Longrightarrow \nu    \ \ \ \ \mbox{ in } M_{ +}(\mathbb{R}).
\]
Now, one can easily establish the identity between the empirical mean excess function and the empirical measure
\[ 
 	\hat M (X_{ (\lceil ku \rceil)})   = \frac{ kc_{ \lceil n/k \rceil}}{ \lceil ku \rceil-1}\int_{ \frac{ X_{ (\lceil ku \rceil)}-X_{ (k)}}{ c_{ \lceil n/k \rceil}}}^{ \infty } \hat \nu_{ n}(x,\infty )dx.
\]
From this fact it follows that 
\[ 
 	S_{ n}(u)= \left( \frac{ X_{ (\lceil ku \rceil)}-X_{ (k)}}{ c_{ \lceil n/k \rceil} }, \frac{  \hat M(X_{ (\lceil ku \rceil)})  }{  c_{ \lceil n/k \rceil}} \right)    \stackrel{ P}{\longrightarrow  } S(u)= (-\ln u,1) 
\]
in $ D^{ 2}_{ l}(0,1]$ and that in turn implies
 \[ 
 	\Big\{ \Big( \frac{ X_{ (i)}-X_{ (k)}}{ c_{ \lceil n/k \rceil}}, \frac{ \hat M(X_{ (i)})}{c_{ \lceil n/k \rceil}} \Big) :1\le i\le k   \Big \}   \stackrel{ P}{ \longrightarrow} \Big\{ \big( t, 1\big):0\le t< \infty  \Big\} 
\]
Finally, using the fact that 
\[ 
 	\frac{ X_{ (\lceil k/2 \rceil)}-X_{ (k)}}{ c_{ \lceil n/k \rceil} }   \stackrel{ P}{ \longrightarrow  } \ln 2 
\]
we get the desired result.
\end{proof}

\end{subsection}

\end{section}

\begin{section}{Comparison with Other Methods of Extreme Value Analysis}\label{sec:exmethod}
For iid random variables from a distribution in the
maximal domain of attraction of the Frechet, Weibull or the Gumbel
distributions, Theorems \ref{thm:positive:xi}, \ref{thm:negative:xi} and
\ref{thm:zero:xi} describe the asymptotic behavior of the ME plot for
high thresholds.  {Linearity of the ME plot  for high order statistics
indicates there is no evidence against the hypothesis that
 the GPD model is a good fit for the thresholded data. }

{Furthermore, we obtain
a natural estimate $\hat \xi$ of $\xi$ by fitting a
 line to the linear part of the ME plot using least squares to get a
 slope estimate $\hat b$ and then
 recovering $\hat \xi= \hat b/(1+\hat b)$. Although natural, 
 convergence of the ME plot to a linear limit does not
 guarantee consistency of this estimate $\hat \xi $ and this is still
 under consideration.  Proposition 5.1.1 in
\cite{das2008qpr} explains why the slope of the least squares line is
not a continuous functional of  finite random sets.}

{ \cite{davison1990meo} give another method to estimate $\xi$.
They suggest a way to find a threshold
using the ME plot and then fit a GPD to the points above the
threshold using maximum likelihood estimation. For both this and the
LS method, the ME plot obviously plays a central role.
We analyze several  simulation and real data sets in Sections \ref{sec:simul} and
\ref{sec:data} using only the LS  method. }

{With any method, an important step is choice of threshold guided
  by the
ME plot so that the plot is {roughly}  linear above this
threshold. Threshold choice can be  challenging
and parameter estimates can be sensitive to the
threshold choice, especially when real data is analyzed.}

{The ME plot is only one of a suite of widely used tools}
 for extreme value model selection. Other
techniques are the Hill plot, the Pickands plot, the moment estimator
plot and the QQ plot; cf.  Chapter 4, \cite{resnick2006htp} and \cite{dehaan:ferreira:2006}.
{Some comparisons from the point of
view of asymptotic bias} and variance are  in
\cite{dehaan:peng:1998}. Here we review definitions and basic facts about
several methods assuming that $X_{ 1},\ldots ,X_{ n} $ is an iid
sample from a distribution in the maximal domain of attraction of an
extreme value distribution. The asymptotics require $k= k_{ n}$, the
number of upper order statistics used for estimation,  to be a
sequence increasing to $ \infty$ such that $ k_{ n}/n\to 0$. 

\begin{enumerate}[(1)] 
        \item  The Hill estimator {based on $m$ upper order statistics} is
\[ 
 	H_{ m,n}= \Big( \frac{ 1}{m }\sum_{ i=1}^{ m}\log \frac{ X_{
            (i)}}{X_{ (m+1)} } \Big)^{ -1} ,\qquad
 1\le m\le n.
\]
If $ \xi>0$ then  $H_{ k_{ n},n} \stackrel{ P}{\longrightarrow } \alpha = 1/\xi$. The Hill plot is the plot of the points $ \{(k,H_{ k,n}):1\le k\le n\}$.  

\item The Pickands estimator does not impose any restriction on the
  range of $ \xi$.   The Pickands estimator,
\[ 
 	\hat \xi_{ m,n}= \frac{ 1}{\log 2 }\log \Big( \frac{ X_{
            (m)}-X_{ (2m)}}{ X_{ (2m)}-X_{ (4m)}} \Big),\qquad 1\le m\le [n/4], 
\]
is consistent for $ \xi\in\mathbb{R}$; i.e., $ \hat \xi_{ k_{ n},n} \stackrel{ P}{\longrightarrow  } \xi  $ as $ n \to \infty$. The Pickands plot is the plot of the points $ \{(k,\hat \xi_{ m,n}),1\le m\le [n/4]\}$.
\item The QQ plot treats the case $ \xi>0$ and $ \xi<0$
  separately. When $ \xi>0$,
{the QQ plot consists of }  the points $ \mathcal{Q}_{ m,n}:=\{   \big(
 -\log(i/m),\log(X_{ (i)}/X_{ (m)}) \big):1\le i\le m   \}$ where $
 m<n$ is a suitably chosen integer. \cite{das2008qpr} showed  $
 \mathcal{Q}_{ k_{ n},n}\to\{   (t,\xi t):t\ge 0  \}$ in $ \mathcal{F}
 $ equipped with the Fell topology.
So the limit is a line with slope
 $ \xi$  and the LS estimator is consistent \citep{das2008qpr,  kratz1996}.

In the case when $ \xi<0$ then the QQ plot can be defined as the plot
of the points $ \mathcal{Q}^{ \prime }_{ m,n}:= \{   \big( X_{ (i)},
G^{ \leftarrow } _{ \hat \xi,1} (i/(n+1)) \big):1\le i\le n   \}$,
where $ \hat \xi$ is an estimate of $ \xi$ {based on $m$ upper
  order statistics.}

\item The moment estimator 
\citep{dekkers:einmahl:dehaan:1989, dehaan:ferreira:2006}
is another method which works for all $ \xi \in \mathbb{R}$ {and} is defined as 
\[ 
 	\hat \xi ^{ (moment)}_{ m,n}=H^{ (1)}_{ m,n}+1- \frac{ 1}{ 2} \Bigg( 1- \frac{ (H_{ m,n}^{ (1)})^{ 2}}{H_{ m,n}^{ 2} } \Bigg) ^{ -1} ,\qquad 1\le m\le n,
\]
where 
\[ 
 	H^{ (r)}_{ m,n}= \frac{ 1}{ m} \sum \Bigg( \log \frac{ X_{
            (i)}}{ X_{ (m+1)} } \Bigg)^{ r} , \qquad  r=1,2.
\]
The moment estimator plot is the plot of the points $ \{
(k,\hat\xi^{ (moment)}_{ {k},n}):1\le k\le n  \}$.
The moment estimator is consistent for $ \xi$.

\item To complete this survey, recall that the ME plot {converges to a
nonrandom  line}
when $ \xi<1$.
\end{enumerate}

{The Hill and QQ plots work best for $\xi>0$ and the ME plot requires
knowledge that $\xi<1$.  Each plot requires the data be properly
thresholded. The ME plot requires thresholding but also that $k$ be
sufficiently large that proper averaging takes place.}

\end{section}

\begin{section}{Simulation}\label{sec:simul}
{We divide this section into three subsections. In  subsection
  \ref{subsec:standard} we show simulation results for mean excess
  plot of some standard distributions with well-behaved tails.   In
  subsections \ref{subsub:diff} and \ref{subsub:infmean} we discuss
  simulation results of some distributions with either difficult
  tail-behavior or infinite mean.} 

\subsection{Standard Situations}\label{subsec:standard}
\subsubsection{Pareto distribution}
The obvious first choice for a distribution function to simulate from is
the GPD. For the GPD the ME plot should be {roughly} linear. We
simulate 50000 random variables from the Pareto(2) distribution. This
means that the parameters of the GPD are $ \xi=0.5$ and $ \beta =1$.
\begin{figure}[t]\label{fig:pareto2}
\centering
\caption{ME plot $\bigl\{ \bigl(X_{(i)} ,\hat M(X_{(i)})\bigr),1\leq i
  \leq 50000\bigr\}$  of 50000 random variables from Pareto(2) distribution $ (\xi=0.5)$. (a) Entire plot, (b) Order statistics 250-50000.}
\label{fig:pareto}
\includegraphics[width=12.5cm]{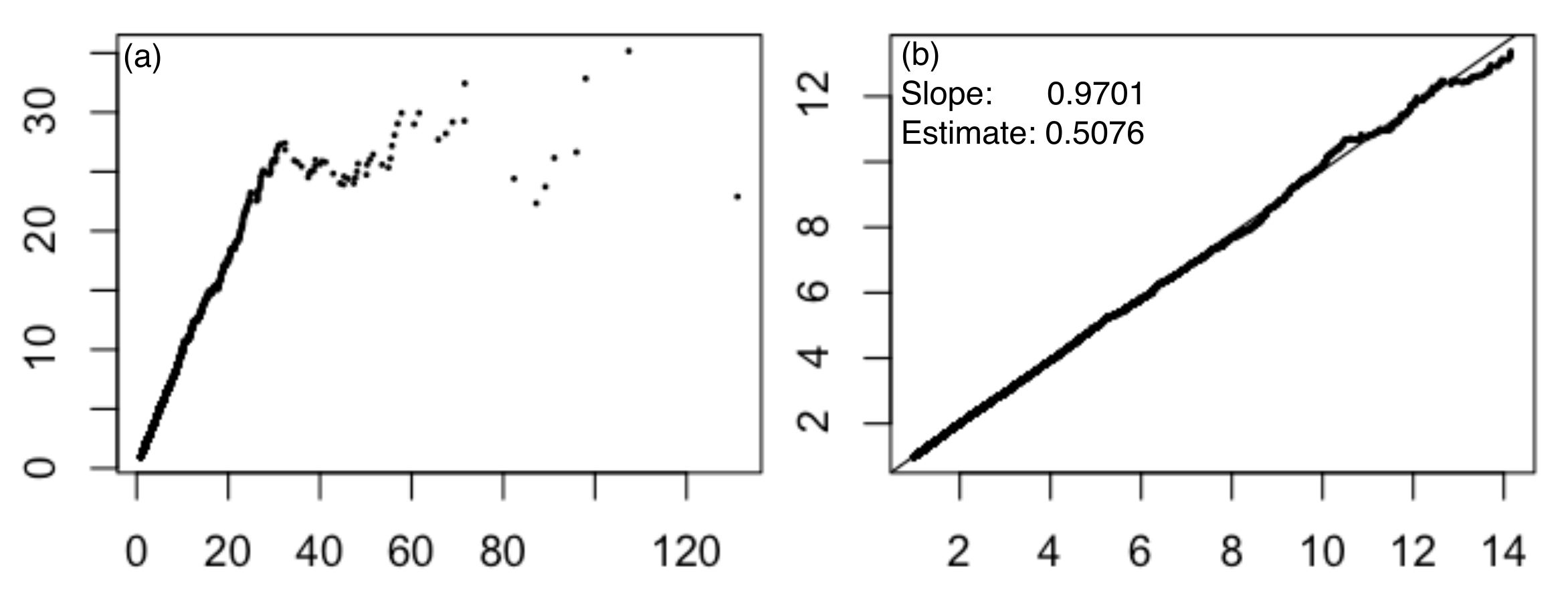}
\end{figure}
Figure \ref{fig:pareto} shows the mean excess plot for this data
set. Observe that in Figure \ref{fig:pareto}(a) the first part of the
plot is quite linear but it is scattered for very high order
statistics. The reason behind this phenomenon is that the empirical
mean excess function for high thresholds is the average of the
excesses of a small number of upper order statistics. When averaging
over few numbers, there is high variability and therefore, this part
of the plot appears very non-linear and is uninformative. In Figure
\ref{fig:pareto}(b) we zoom into the plot by leaving out the top 250
points. {We calculate   using all the data but  plot only  the points} $ \{(X_{
  (i)},\hat M(X_{ (i)})):250\le i\le 50000     \}$. This restricted
plot looks  linear. We fit a least squares line to this plot and the
estimate of the slope is $ 0.9701.$ Since the slope is $ \xi/(1-\xi)$
we get the estimate of $ \xi$ to be $ 0.5076.$ 
\begin{figure}[h]
\centering
\caption{ME plot  of 50000 random variables from totally right skewed Stable(1.5) distribution $ (\xi=2/3)$. (a) Entire plot, (b) Order statistics 120-30000, (c) 180-20000, (d) 270-10000.}
\includegraphics[width=12.6cm]{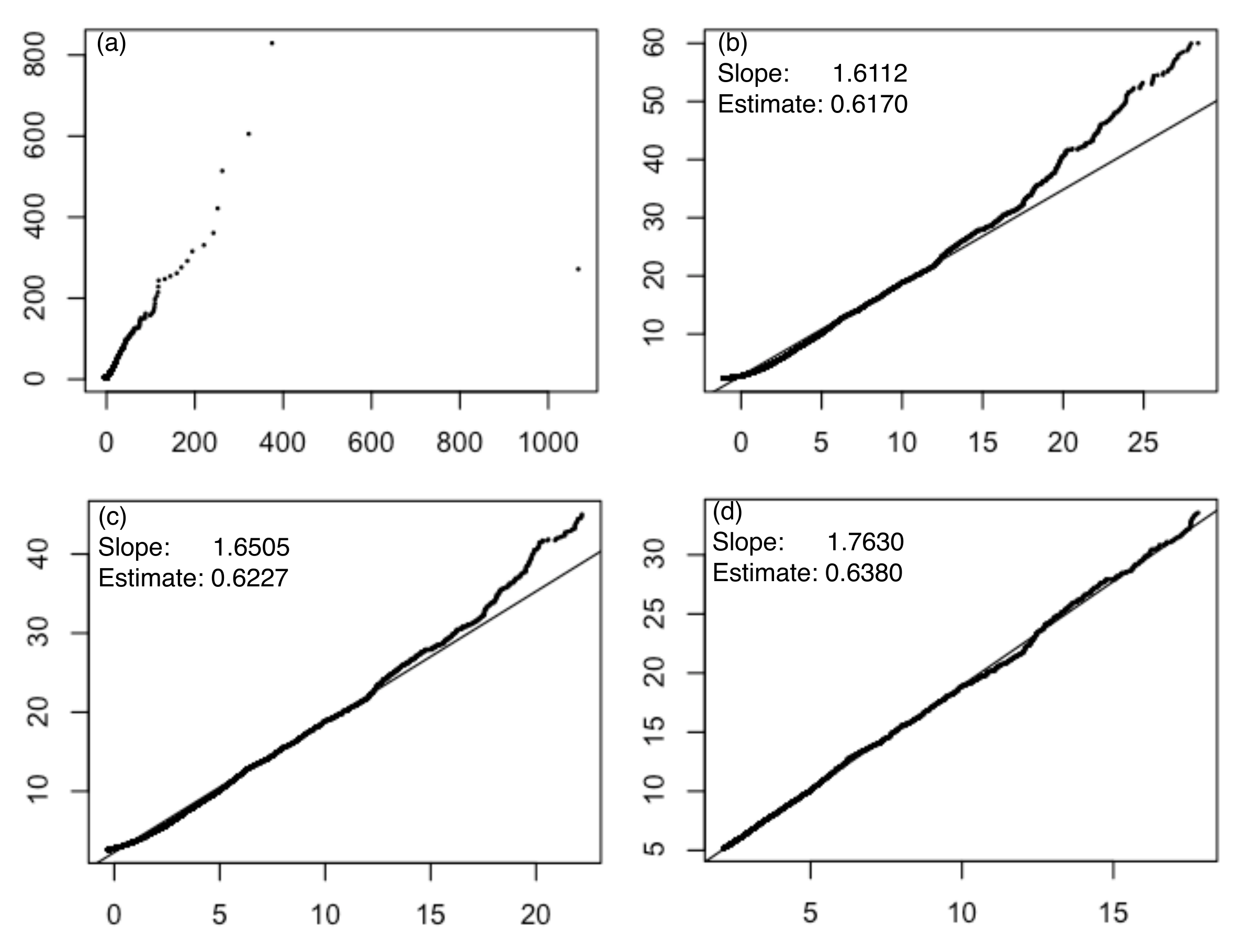}
\label{fig:stable}
\end{figure}

\subsubsection{Right-skewed stable distribution}
We next {simulate} 50000  random samples  from a totally right
skewed stable(1.5) distribution. So $ \bar F\in RV_{ -1.5}$ and then $
\xi=2/3$. Figure \ref{fig:stable}(a) is the ME plot obtained from this
data set.  This is not a sample from a GPD, but {only} from a distribution in
the maximal domain of attraction of a GPD. The  ME function is
not exactly linear and for estimating $\xi$ we should concentrate on  high
thresholds. As we did for the last example we  drop  points in the
plot for very high order statistics since they are the average of a
very few {values}. Figures \ref{fig:stable}(b), \ref{fig:stable}(c) and
\ref{fig:stable}(d) confines the plot to the order statistics
120-30000, 180-20000 and 270-10000 respectively, i.e., plots the
points $ (X_{ (i)},\hat M(X_{ (i)}))$ for $ i$ in the specified
range. As we restrict the plot more and more, the plot becomes
increasingly linear. In Figure \ref{fig:stable}(d) the least squares
estimate of the slope of the line is 1.763 and hence the estimate of $
\xi$ is 0.638.

\subsubsection{Beta distribution}Figure \ref{fig:beta} gives the ME plot for 50000 random variables from the beta(2,2) distribution
\begin{figure}[h]
\centering
\caption{ME Plot of $ 50000$ random variables from the beta(2,2) distribution $ (\xi=-0.5)$. (a) Entire plot, (b) Order statistics: 150-35000, (c) 300-20000, (d) 450-5000. }
\includegraphics[width=12.6cm]{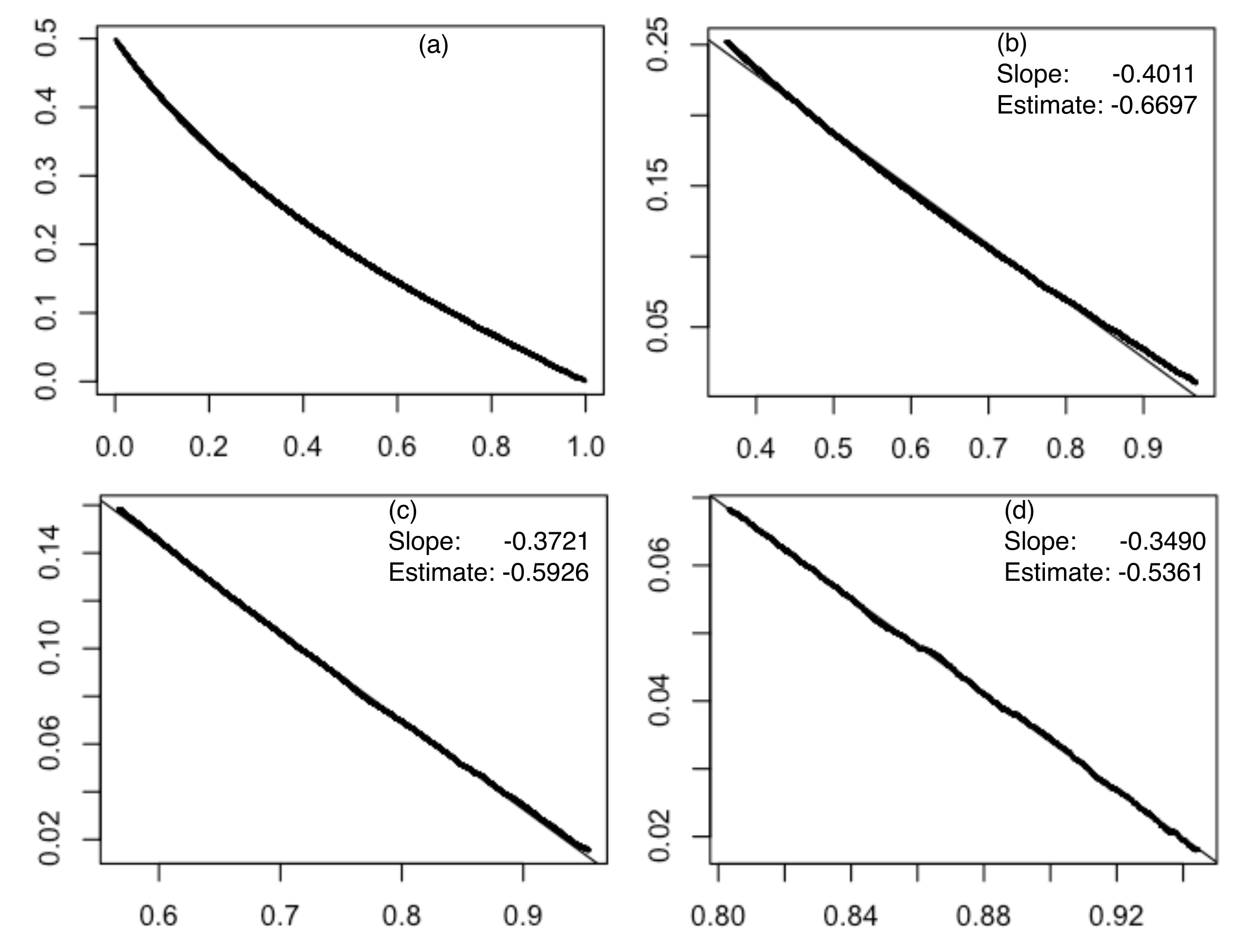}
\label{fig:beta}
\end{figure}
which is in the maximal domain of the Weibull distribution with the parameter $ \xi=-0.5$. Figure \ref{fig:beta}(a) is the entire ME plot and then Figures \ref{fig:beta}(b), \ref{fig:beta}(c) and \ref{fig:beta}(d) plot the empirical ME function for the order statistics 150-35000, 300-20000 and 450-5000 respectively. The last plot is quite linear and the estimate of $ \xi$ is $ -0.5361$.

\subsection{Difficult Cases}\label{subsub:diff}

\subsubsection{Lognormal distribution}The lognormal(0,1) distribution is in the maximal domain of attraction of the Gumbel and hence $ \xi=0$. The ME function of the log normal has the form
\[ 
 	M(u)= \frac{ u}{\ln u }(1+o(1)) \ \ \ \ \mbox{ as } u\to \infty;
\]
see Table 3.4.7 in \cite[p.161]{embrechtskluppelbergmikosch:1997}. So
$ M(u)$ is regularly varying of index $ 1$ but still $ M^{ \prime
}(u)\to 0$. Figure \ref{fig:lognormal}(a) shows the ME plot obtained
for a sample {of size} $ 10^{ 5}$  from the lognormal(0,1) distribution. 
\begin{figure}[h]
\centering
\caption{ME plot of $ 10^{ 5}$ random variables from the lognormal(0,1) distribution $ (\xi=0)$. (a) Entire plot, (b) Order statistics: 150-70000, (c) 300-40000, (d) 450-10000. }
\label{fig:lognormal}
\includegraphics[width=12.6cm]{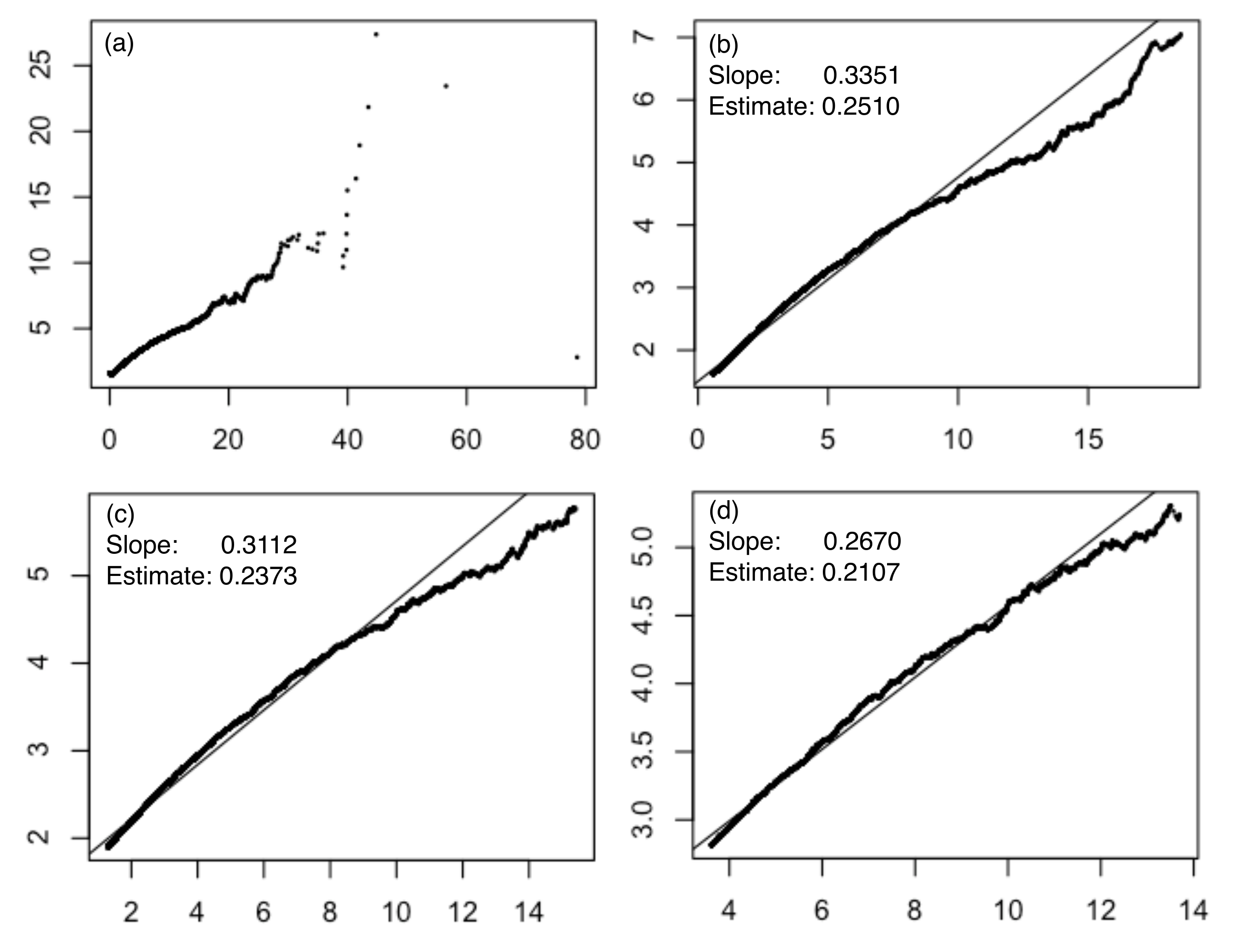}
\end{figure}
Figures \ref{fig:lognormal}(b), \ref{fig:lognormal}(c) and
\ref{fig:lognormal}(d) {show} the empirical ME functions for the order
statistics 150-70000, 300-40000 and 450-10000 respectively. The slopes
of the least squares lines in Figures \ref{fig:lognormal}(b),
\ref{fig:lognormal}(c) and \ref{fig:lognormal}(d) are   0.3351, 0.3112
and 0.267 respectively. The estimate of $ \xi$ also decreases steadily
as we zoom in towards the higher order statistics from 0.251 in
\ref{fig:lognormal}(b) to 0.2107 in \ref{fig:lognormal}(d).
Furthermore, a curve is evident in the plots and the slope of the
curve is decreasing, albeit very slowly, as we look at higher and
higher thresholds. At a first glance the ME function might seem to
resemble that of a distribution in the maximal domain of attraction of
the Frechet. The curve becomes evident only after a detailed analysis
of the plot. That is possible because 
{the data are simulated but in practice analysis would be
  difficult.}
For this example, the ME plot is not 
a very effective diagnostic for discerning the model.

\subsubsection{A non-standard distribution}We also try a non-standard distribution for which $ \bar F ^{ -1}(x)=x^{ -1/2}(1-10\ln x), 0<x\le1$. This means that $ \bar F\in RV_{ -2}$ and therefore $ \xi=0.5$. The exact form of $ \bar F$ is given by 
\begin{equation}\label{eq:nonstd} 
 	\bar F(x)=400 W\big( xe^{ 1/20}/20 \big)^{ 2} x^{ -2}    \ \ \ \ \mbox{ for all } x\ge 1,
\end{equation}
where $ W$ is the Lambert W function satisfying $ W(x)e^{ W(x)}=x$ for all $ x>0$.  Observe that $ W(x)\to \infty$ as $ x\to \infty$ and $ W(x)\le \log (x)$ for $ x>1$. Furthermore,
\[ 
 	\frac{ \log (x)}{W(x) } =1+ \frac{ \log W(x)}{W(x) } \to 1 \ \ \ \ \mbox{ as  } x\to \infty,   
\]
and hence $ W(x)$ is a slowly varying function.  This is therefore an
example where the slowly varying term contributes significantly to $
\bar F$. That was not the case in the Pareto or  the stable examples.  
\begin{figure}[h]
\centering
\caption{ME Plot of $ 10^{ 5}$ random variables from the  distribution  in \eqref{eq:nonstd}. $ (\xi=0.5)$. (a) Entire plot, (b) Order statistics: 150-70000, (c) 400-20000, (d) Hill Plot estimating $ \alpha =1/\xi$. }
\label{fig:nonstand}
\includegraphics[width=12.6cm]{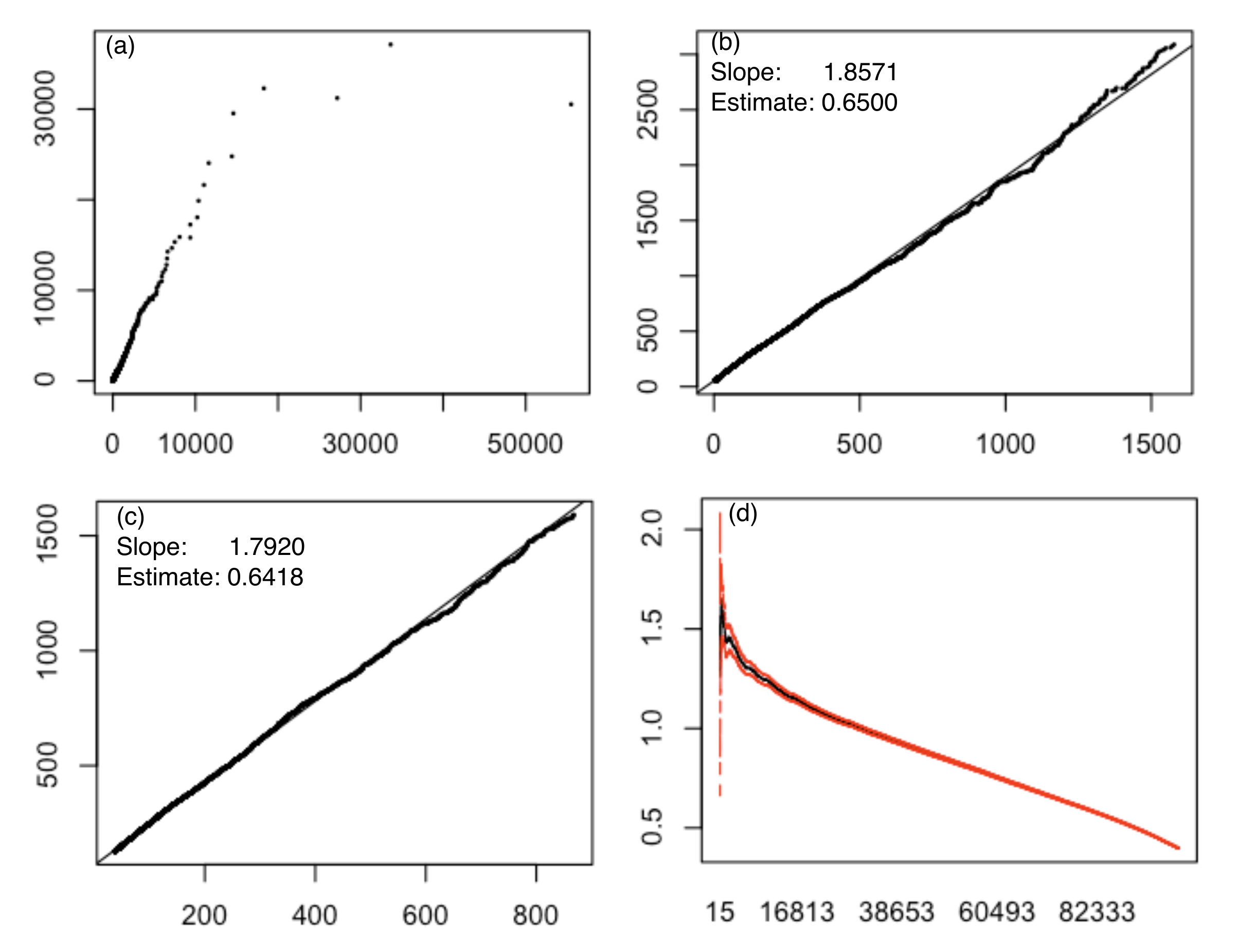}
\end{figure}
We simulated $ 10^{ 5}$ random variables from this
distribution. Figure \ref{fig:nonstand}(a) gives the entire ME plot
from this data set. Figures \ref{fig:nonstand}(b) and
\ref{fig:nonstand}(c) plots the ME function for the order statistics
150-70000 and 400-20000 respectively. In Figure \ref{fig:nonstand}(c)
the estimate of $ \xi$ is 0.6418 which 
{is a somewhat disappointing estimate given
that the sample size was  $ 10^{ 5}$}. Figure
\ref{fig:nonstand}(d) is the Hill Plot from this data set using the
\emph{QRMlib} package in R. It plots the estimate of $ \alpha =1/\xi$
obtained by choosing  different values of $ k$. It is evident from
this that the Hill estimator does not perform well here. For none of
the values of $ k$ is the Hill estimator even close to the true value
of $ \alpha $ which is 2. We conclude, {not surprisingly}, that a slowly varying function
increasing to infinity can fool both the ME plot and the Hill
plot. See
 \cite{Degen:2007p5866} for a discussion on the behavior of the ME
 plot for a sample simulated from the g-and-h distribution and
 \cite{resnick2006htp} for {\it Hill horror plots\/}.

\subsection{Infinite Mean: Pareto with $
  \xi=2$.}\label{subsub:infmean} This simulation  
\begin{figure}[h]
\centering
\caption{ME plot of 50000 random variables from Pareto(0.5) distribution $ (\xi=2)$. (a) Entire plot, (b) Order statistics 250-10000.}
\label{fig:pareto0-5}
\includegraphics[width=12.5cm]{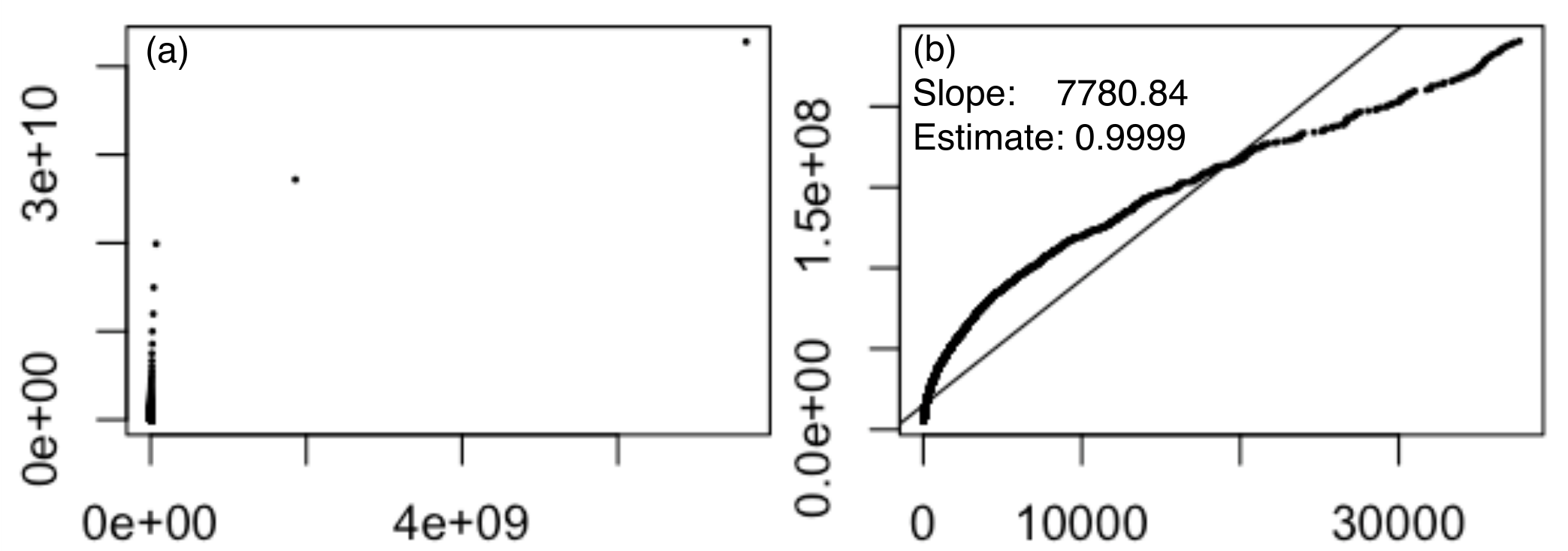}
\end{figure}
{sheds light on the behavior of the ME plot} when $ \xi>1$. In
this case the ME function does not exist but {the empirical ME
  plot does}. Figure \ref{fig:pareto0-5} 
{displays} the ME plot of for 50000 random variables simulated from
Pareto(0.5) distribution.  
The plot is certainly far from linear even for high order statistics and the least squares line has slope 7780.84 which gives an estimate of $ \xi$ to be 0.9999. This certainly gives an indication that the ME plot is not a good diagnostic in this case. 
\end{section}

\begin{section}{ME Plots for Real Data}\label{sec:data}

\subsection{Size of Internet Response} 
{This data set consists of  Internet response sizes  corresponding
to user  requests. The sizes are thresholded to be 
at least 100KB. The data set a part of a bigger set
collected in April 2000 at the University of North Carolina at Chapel
Hill. }

Figure \ref{fig:internet} contains various  plots of the data.  Figures
\ref{fig:internet}(b) and Figure \ref{fig:internet}(e) are the Hill
plot (estimating $ 1/\xi$) and the Pickands plot respectively. It is
difficult to infer anything from these plots  though superficially
they appear stable.  
\begin{figure}[h]
\centering
\caption{Internet response sizes.  (a) Scatter plot, (b) Hill Plot
  estimating $ \alpha =1/\xi$, (c) ME plot, (d) ME plot for order
  statistics 300-12500, (e) Pickands Plot for $ \xi$, (f) QQ plot with
  $ k=15000$, (g) QQ plot with $ k=5000$. } 
\includegraphics[width=14.2cm]{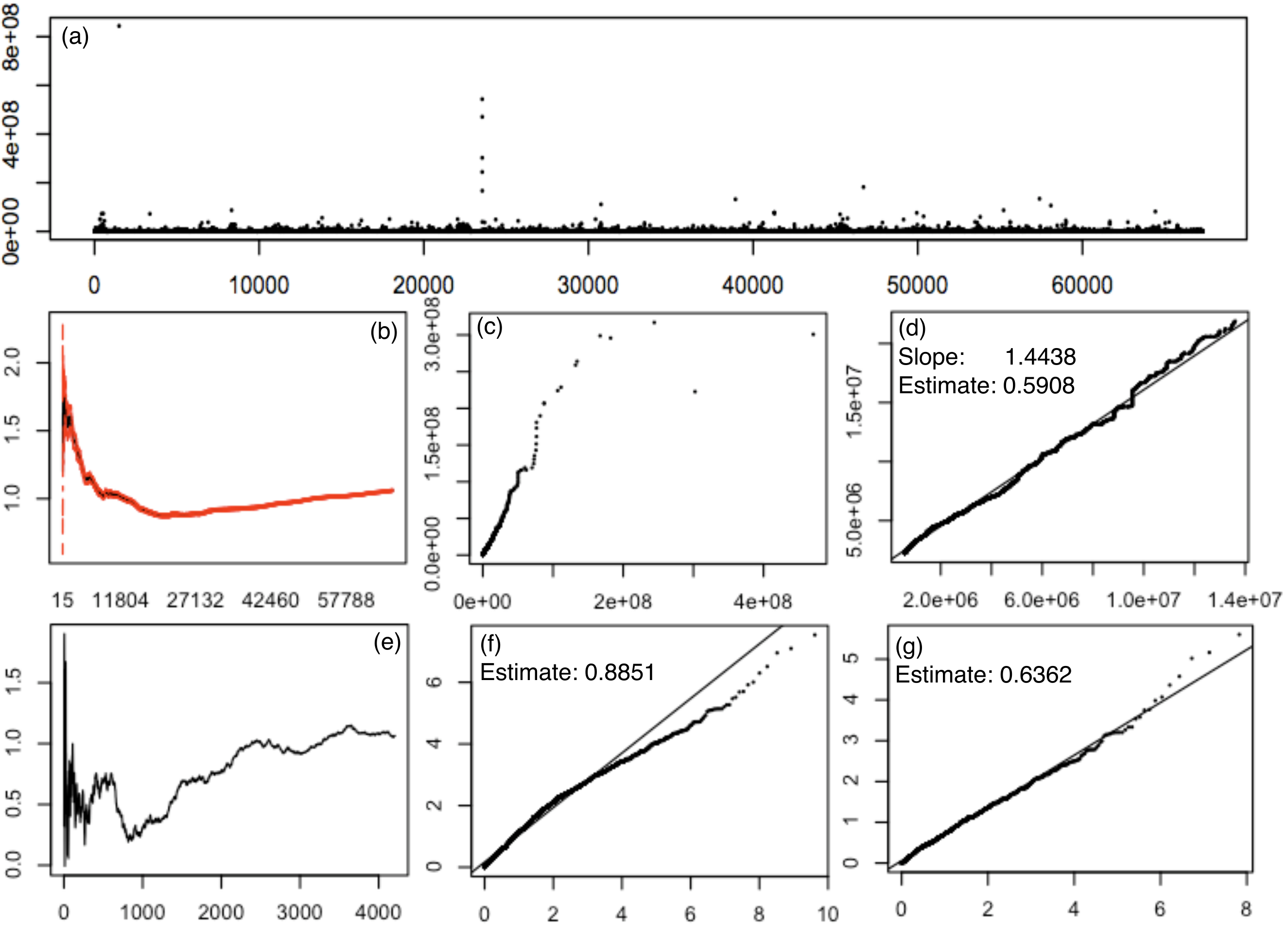}
\label{fig:internet}
\end{figure}
Figures \ref{fig:internet}(c) and \ref{fig:internet}(d) are the entire
ME plot and the ME plot restricted for order statistics 300-12500. The
second plot does seem to be very linear and gives an estimate of $
\xi$ to be 0.5908. Figures \ref{fig:internet}(f) and
\ref{fig:internet}(g) are the QQ plots for the data for $ k=15000$ and
$ k=2500$ (as explained in Section \ref{sec:exmethod}). The estimate
of $ \xi$ in these two plots are 0.8851 and 0.6362. The estimates of $
\xi$ obtained from the QQ plot \ref{fig:internet}(d) and the ME plot
\ref{fig:internet}(g) are close and the plots are also linear. So we {believe}
 that this is a reasonable estimate of $ \xi$. 

\begin{figure}[t]
\centering
\caption{Daily discharge of water in Hudson river. (a) Time series plot, (b) Homoscedasticized plot, (c) Residual plot, (d) ACF of residuals, (e) Hill plot for $ \alpha =1/\xi$, (f) ME plot, (g) ME plot for order statistics 300-1300, (h) Pickands plot, (i) QQ plot with $ k=8000$, (j) QQ plot with $ k=600$.}
\label{fig:hudson}
\includegraphics[width=14.2cm]{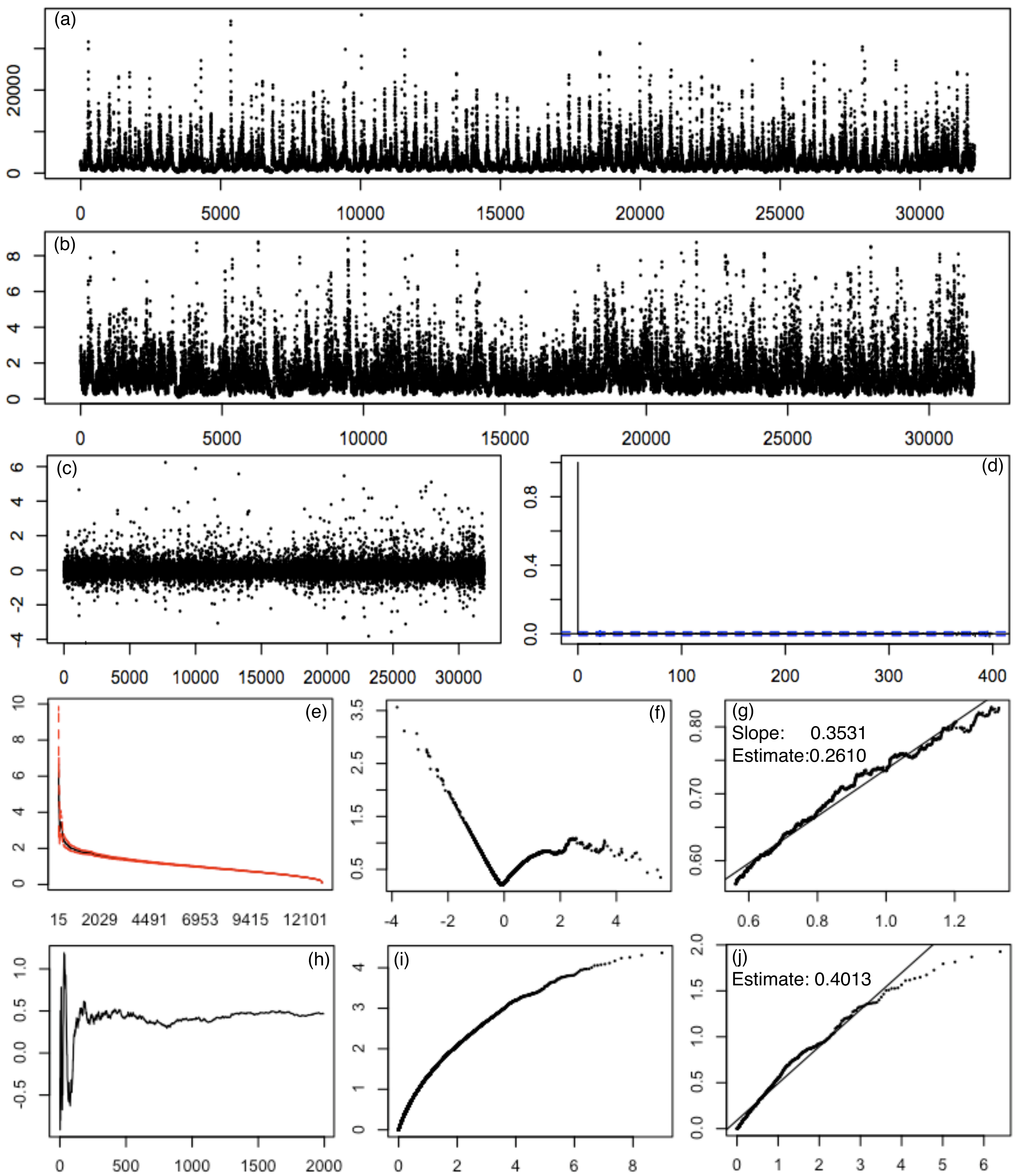}
\end{figure}

\subsection{Volume of Water in the Hudson River}\label{ssec:hudson} We
now analyze data  on the average daily discharge of water (in cubic
feet per second) in the Hudson river measured at the U.S. Geological
Survey site number 01318500 near Hadley, NY. The  range of the data is
from July 15, 1921 to December 31, 2008 for a total of 31946 data points.

\begin{figure}[t]
\centering
\caption{Ozone level in New York City. (a) Time series plot, (b) Homoscedasticized plot, (c) Residual plot, (d) ACF of residuals, (e) Hill plot for $ \alpha =1/\xi$, (f) ME plot, (g) ME plot for order statistics 300-1300, (h) Pickands plot, (i) QQ plot with $ k=4000$, (j) QQ plot with $ k=550$.}
\label{fig:ozone}
\includegraphics[width=14cm]{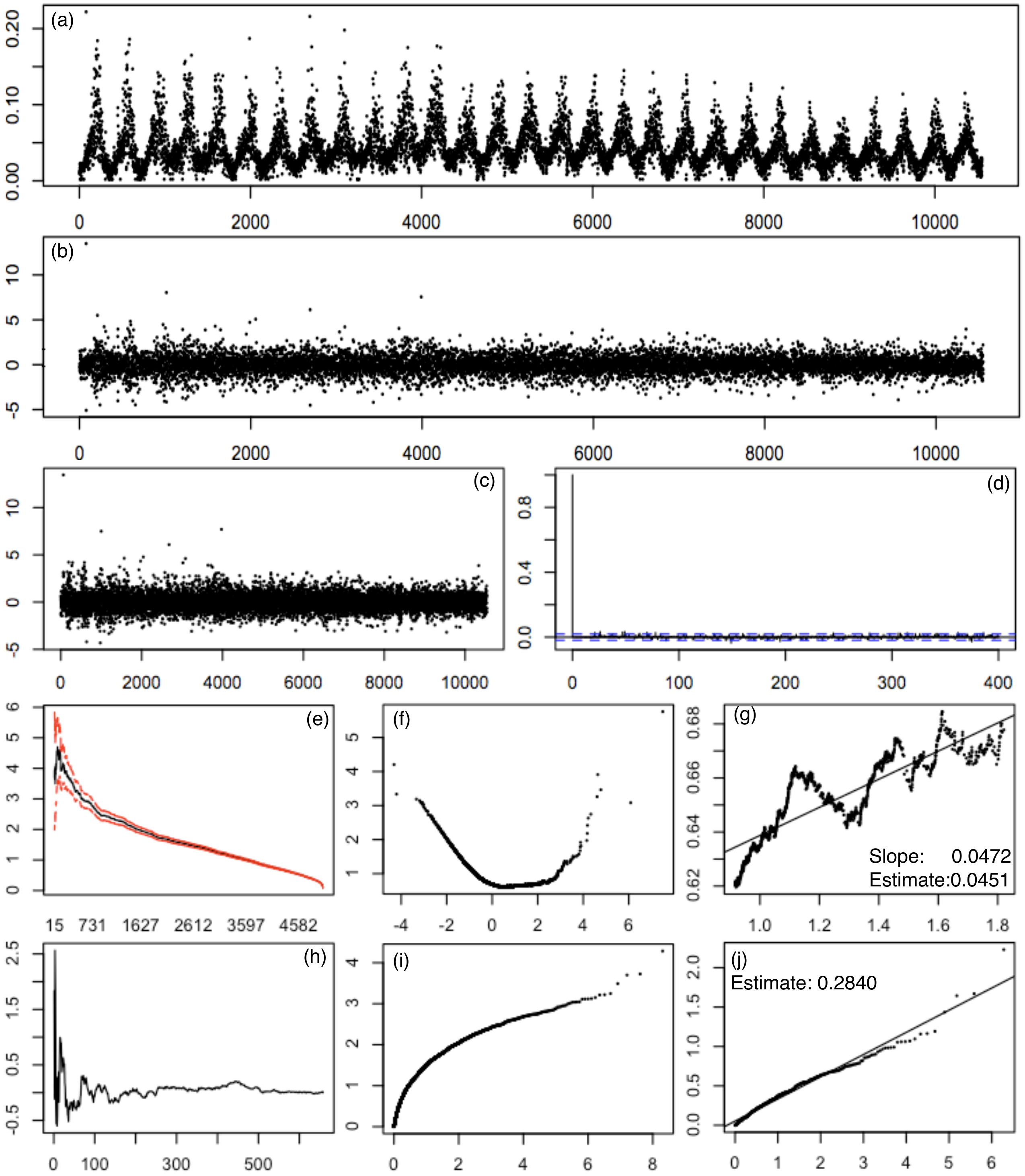}
\end{figure}

Figure \ref{fig:hudson}(a) is the time series plot of the original
data and it  shows the presence of periodicity in the data. The volume
of water is typically much higher in April and May than the rest of
the year which {possibly  is due to snow  melt.}  We
`homoscedasticize' the data in the following way. We compute the
standard deviation of the average discharge of water for every day of
the year and then divide each data point by the standard deviation
corresponding to that day. If the original data is say $ (X_{
  7/15/1921},\cdots,X_{ 12/31/2008})$ then we transform it to $ (X_{
  7/15/1921}/S_{ 7/15},\cdots,X_{ 12/31/2008}/S_{ 12/31})$, where $
S_{ 7/15}$ is the standard deviation of the data points obtained on
July 15 in the different years in the range of the data and similarly
$ S_{ 12/31}$ is the same for December 31. The plot of the transformed
points {is}  given in \ref{fig:hudson}(b). We then fit an AR(33)
model to 
this data using the function \emph{ar} in the \emph{stats} package in
R. The lag was chosen {based on the AIC criterion}. 
{Figures \ref{fig:hudson}(c) and (d) show the residuals and their ACF
plot respectively.} This encourages us to
assume that there no linear dependence in the residuals. 

 We now apply the tools for extreme value analysis on the
 residuals. Figure \ref{fig:hudson}(e) is the Hill plot and it is
 difficult to draw any inference from this plot in this case.
 Figures \ref{fig:hudson}(f) and \ref{fig:hudson}(g) are the entire ME
 plot and the ME plot restricted to the order statistics
 300-1300. From \ref{fig:hudson}(g) we get an estimate of $ \xi$ to be
 0.261. The Pickands plot in \ref{fig:hudson}(h) and the QQ plots in
 \ref{fig:hudson}(i) and \ref{fig:hudson}(j) suggest an estimate of $
 \xi$ around 0.4. A definite curve is visible in the QQ plot even for
 $ k=600$. But the slope of the least squares line fitting the QQ plot
 supports the estimate suggested by the Pickands plot and the ME
 plot. {We see that it is difficult to reach a conclusion about the range of $ \xi$. Still we infer that 0.4 is a reasonable estimate of $ \xi$ for this
 data since that is being suggested by  two different methods.}

\subsection{Ozone level in New York City} We also apply the methods to
a data set obtained from
\url{http://www.epa.gov/ttn/airs/aqsdatamart}. This is the data on
daily maxima of level of Ozone (in parts per million) in New York City
on measurements closest to the ground level  observed between January
1, 1980 and December 31, 2008.

Figure \ref{fig:ozone}(a) is the time series plot of the data. This
data set also showed a seasonal component which accounted for high
values during the summer months.  We transform the data set to a
homoscedastic series (Figure \ref{fig:ozone}(b)) using the same
technique as explained in Subsection \ref{ssec:hudson}. Fitting an
AR(16) model we get the residuals which are uncorrelated; see Figures
\ref{fig:ozone}(c) and \ref{fig:ozone}(d). 

The Hill plot in Figure \ref{fig:ozone}(e) again fails to give a
reasonable estimate of the tail index. The ME plots in Figures
\ref{fig:ozone}(e) and \ref{fig:ozone}(g) are also very rough.  Figure
\ref{fig:ozone}(g) is the plot of the points $ (X_{ (i)},\hat M(X_{
  i}))$ for $ 300\le i\le 1300$ and the least squares line fitting
these points has slope 0.0472 which gives an estimate of $ \xi$ to be
0.0451. This is consistent with the Pickands plot in
\ref{fig:ozone}(h). This suggests that the residuals may be in the
domain of attraction of the Gumbel distribution.  
\end{section}

\section{Conclusion.}
The ME plot may be used as a diagnostic to aid in tail or quantile
estimation for risk management and other extreme value problems. 
However, some problems  associated with its use certainly exist:
\begin{itemize}
\item  One needs to trim away
$\{(X_{(i)},\hat M(X_{(i)}))\}$ for small values of $i$ where too few
terms are averaged and also trim  irrelevent terms for large values of $i$ which
are governed by either the center of the distribution or the left
tail. So two discretionary cuts to the data need be made whereas for
other diagnostics only one threshold needs to be selected.
\item The analyst needs to be convinced  $\xi <1$ since for $\xi\geq
  1$ random sets are the limits for the normalized ME plot. Such
  random limits could create  misleading impressions. The Pickands and moment estimators place no such restriction on the range of $\xi$. The QQ method works most easily  when $\xi >0$ but can be extended to all $\xi \in \mathbb{R}$. The Hill method requires $\xi>0$.
\item Distributions not particularly close to GPD can fool the ME diagnostic.
However, fairness requires pointing out that  this is true of all the
procedures in the extreme value catalogue. In particular, with heavy
tail distributions, if a slowly varying factor is attached to a Pareto
tail, diagnostics typically perform poorly.
\end{itemize}

The standing assumption for the proofs in this paper is that $ \{X_{ n}\}$ is iid. 
We believe most of the results on the ME plot hold under the assumption that the 
underlying sequence $ \{X_{ n}\}$ is stationary and the tail empirical measure is consistent for the limiting GPD distribution
of the marginal distribution of $ X_{ 1}$. We intend to look into this
further. Other open issues engaging our attention include converses to
the consistency of the ME plot and if the slope of the least squares
line through the ME plot is a consistent estimator. 

We are thankful to the referees and the editors for their valuable  and detailed comments. 

\bibliographystyle{elsart-harv}
\bibliography{/Users/Souvik/Documents/Chronicles/Miscellanea/Latex/bibfile}

\begin{thebibliography}{27}
\expandafter\ifx\csname natexlab\endcsname\relax\def\natexlab#1{#1}\fi
\expandafter\ifx\csname url\endcsname\relax
  \def\url#1{\texttt{#1}}\fi
\expandafter\ifx\csname urlprefix\endcsname\relax\def\urlprefix{URL }\fi

\bibitem[{Benktander and Segerdahl(1960)}]{Benktander:1960}
Benktander, G., Segerdahl, C., 1960. On the analytical representation of claim
  distributions with special reference to excess of loss reinsurance. In: XVIth
  International Congress of Actuaries, Brussels.

\bibitem[{Billingsley(1999)}]{billingsley1999cpm}
Billingsley, P., 1999. Convergence of Probability Measures, 2nd Edition. John
  Wiley and Sons, New York.

\bibitem[{Bingham et~al.(1989)Bingham, Goldie, and
  Teugels}]{bingham:goldie:teugels:1989}
Bingham, N.~H., Goldie, C.~M., Teugels, J.~L., 1989. Regular Variation.
  Cambridge University Press.

\bibitem[{Coles(2001)}]{coles:2001}
Coles, S., 2001. An Introduction to Statistical Modeling of Extreme Values.
  Springer-Verlag, London.

\bibitem[{Csorgo and Mason(1986)}]{csorgo1986ads}
Csorgo, S., Mason, D., 1986. The asymptotic distribution of sums of extreme
  values from a regularly varying distribution. The Annals of Probability
  14~(3), 974--983.

\bibitem[{Das and Resnick(2008)}]{das2008qpr}
Das, B., Resnick, S., 2008. Qq plots, random sets and data from a heavy tailed
  distribution. Stochastic Models 24~(1), 103--132.

\bibitem[{Davison and Smith(1990)}]{davison1990meo}
Davison, A., Smith, R., 1990. {Models for exceedances over high thresholds}.
  Journal of the Royal Statistical Society Series B 52~(3), 393--42.

\bibitem[{de~Haan(1976)}]{DeHaan:1976p6219}
de~Haan, L., 1976. An {A}bel-{T}auber theorem for {L}aplace transforms. Journal
  of the London Mathematical Society 13, 537--542.

\bibitem[{de~Haan and Ferreira(2006)}]{dehaan:ferreira:2006}
de~Haan, L., Ferreira, A., 2006. Extreme Value Theory: An Introduction.
  Springer-Verlag, New York.

\bibitem[{{de Haan} and Peng(1998)}]{dehaan:peng:1998}
{de Haan}, L., Peng, L., 1998. Comparison of tail index estimators. Statistica
  Neerlandica 52~(1), 60--70.

\bibitem[{Degen et~al.(2007)Degen, Embrechts, and Lambrigger}]{Degen:2007p5866}
Degen, M., Embrechts, P., Lambrigger, D.~D., 2007. The quantitative modeling of
  operational risk: between g-and-h and evt. Astin Bulletin 37~(2), 265.

\bibitem[{Dekkers et~al.(1989)Dekkers, Einmahl, and
  de~Haan}]{dekkers:einmahl:dehaan:1989}
Dekkers, A., Einmahl, J., de~Haan, L., 1989. A moment estimator for the index
  of an extreme-value distribution. Ann. Statist. 17, 1833--1855.

\bibitem[{Embrechts et~al.(1997)Embrechts, Kl{\"u}ppelberg, and
  Mikosch}]{embrechtskluppelbergmikosch:1997}
Embrechts, P., Kl{\"u}ppelberg, C., Mikosch, T., 1997. Modelling Extremal
  Events. Vol.~33 of Applications in Mathematics. Springer-Verlag, New York.

\bibitem[{Embrechts et~al.(2005)Embrechts, McNeil, and
  Frey}]{embrechtsmcneilfrey:2005}
Embrechts, P., McNeil, A.~J., Frey, R., 2005. Quantitative Risk Management:
  Concepts, Techniques, and Tools. Princeton University Press.

\bibitem[{Guess and Proschan(1985)}]{guess1985mrl}
Guess, F., Proschan, F., 1985. Mean Residual Life: Theory and Applications.
  Defense Technical Information Center.

\bibitem[{Hall and Wellner(1981)}]{hall1981mrl}
Hall, W., Wellner, J., 1981. Mean residual life. Statistics and Related Topics,
  169--184.

\bibitem[{Hogg and Klugman(1984)}]{hogg1984ld}
Hogg, R., Klugman, S., 1984. Loss Distributions. Wiley, New York.

\bibitem[{Kratz and Resnick(1996)}]{kratz1996}
Kratz, M., Resnick, S., 1996. The qq-estimator and heavy tails. Stochastic
  Models 12~(4), 699--724.

\bibitem[{Matheron(1975)}]{matheron1975rsa}
Matheron, G., 1975. Random Sets and Integral Geometry. Wiley, New York.

\bibitem[{Molchanov(2005)}]{molchanov:2005}
Molchanov, I.~S., 2005. Theory of Random Sets. Springer.

\bibitem[{Resnick(2007)}]{resnick2006htp}
Resnick, S., 2007. Heavy-Tail Phenomena: Probabilistic And Statistical
  Modeling. Springer.

\bibitem[{Resnick and Starica(1998)}]{resnick1998tail}
Resnick, S., Starica, C., 1998. {Tail index estimation for dependent data}. The
  Annals of Applied Probability, 1156--1183.

\bibitem[{Resnick(1987)}]{resnick:1987}
Resnick, S.~I., 1987. Extreme Values, Regular Variation and Point Processes.
  Springer-Verlag, Berlin, New York.

\bibitem[{Smith(1989)}]{smith1989eva}
Smith, R., 1989. {Extreme value analysis of environmental time series: an
  application to trend detection in ground-level ozone}. Statistical Science
  4~(4), 367--377.

\bibitem[{Todorovic and Rousselle(1971)}]{todorovic1971spf}
Todorovic, P., Rousselle, J., 1971. Some problems of flood analysis. Water
  Resources Research 7~(5), 1144--1150.

\bibitem[{Todorovic and Zelenhasic(1970)}]{todorovic1970smf}
Todorovic, P., Zelenhasic, E., 1970. A stochastic model for flood analysis.
  Water Resources Research 6~(6), 1641--1648.

\bibitem[{Yang(1978)}]{yang1978ebf}
Yang, G., 1978. Estimation of a biometric function. The Annals of Statistics
  6~(1), 112--116.

\end{thebibliography}
\end{document}